%% file: RevisionPalitta_Arxiv.tex
\documentclass[smallextended,openbib]{svjour3}       
\smartqed  
\usepackage{graphicx}
\usepackage{amsfonts,amsmath,amssymb,graphicx,caption}
\usepackage{graphicx,epstopdf} 
\usepackage[caption=false]{subfig} %
\usepackage{float}
\makeatletter
\newcommand*\rel@kern[1]{\kern#1\dimexpr\macc@kerna}
\newcommand*\widebar[1]{%
  \begingroup
  \def\mathaccent##1##2{%
    \rel@kern{0.8}%
    \overline{\rel@kern{-0.8}\macc@nucleus\rel@kern{0.2}}%
    \rel@kern{-0.2}%
  }%
  \macc@depth\@ne
  \let\math@bgroup\@empty \let\math@egroup\macc@set@skewchar
  \mathsurround\z@ \frozen@everymath{\mathgroup\macc@group\relax}%
  \macc@set@skewchar\relax
  \let\mathaccentV\macc@nested@a
  \macc@nested@a\relax111{#1}%
  \endgroup
}
\makeatother

\DeclareMathOperator*{\argmax}{argmax}
\DeclareMathOperator*{\argmin}{argmin}

\usepackage{color,tabularx}
\usepackage{algorithmic,algorithm}
\usepackage[ruled,algosection,algo2e,resetcount]{algorithm2e}

\usepackage{caption}
\input{generic}

\usepackage{pgfplots}			
\pgfplotsset{compat=1.3}		

 \usepackage{multirow}
\usepackage{scalerel,stackengine}

\usepackage{geometry}
\newcommand{\RR}{{\mathbb{R}}}
\newcommand{\CC}{{\mathbb{C}}}

\newtheorem{Prop}[theorem]{Proposition}

\newtheorem{Cor}[theorem]{Corollary}
\newtheorem{num_example}{Example}
\newtheorem{Ass}[theorem]{Assumption}

\begin{document}

\title{The projected Newton-Kleinman method for the algebraic Riccati equation
\thanks{Version of \today.}
}


\author{Davide Palitta
}


\institute{ Davide Palitta \at Research Group Computational Methods in Systems
and Control Theory (CSC),
Max Planck Institute for Dynamics of Complex Technical Systems, Sandtorstra\ss{e} 1, 39106 Magdeburg, Germany\\ \email{palitta@mpi-magdeburg.mpg.de}            
}

\date{Received: date / Accepted: date}

\maketitle
\bibliographystyle{siam}

\begin{abstract}
The numerical solution of the algebraic Riccati equation is a challenging task especially for very large problem dimensions. In this paper we present a new algorithm that combines the very appealing computational features of projection methods with the convergence properties of the inexact Newton-Kleinman procedure equipped with a line search. In particular, the Newton scheme is completely merged in a projection framework with a single approximation space so that the Newton-Kleinman iteration is only implicitly performed. Moreover, the line search 
turns out to be exact in our setting, i.e., the existence of a local minimum of the Riccati residual norm along the current search direction is guaranteed and the corresponding minimizer is chosen as step-size.
This property determines a monotone decrease of the Riccati residual norm under some mild assumptions.
Several numerical results are reported to illustrate the potential of our novel approach.



\end{abstract}


\section{Introduction}
We are interested in the numerical solution of the algebraic Riccati equation\footnote{The Riccati equation
is usually reported as $A^TX+XA-XBB^TX+C^TC=0$. Here we prefer to use $A$ in place of $A^T$ for sake of simplicity in 
the presentation of the Krylov subspace approach.}
\begin{equation}\label{eq.Riccati}
 \mathcal{R}(X):=AX+XA^T-XBB^TX+C^TC=0,
\end{equation}
where $A\in\RR^{n\times n}$ is of very large dimension, and $B\in\RR^{n\times p}$ and $C\in\RR^{q\times n}$ are such that
$p+q\ll n$.
This equation is of great interest in many applications, such as linear-quadratic optimal control problems for parabolic PDEs and
balancing based model order reduction of large linear systems. See, e.g., \cite{Antoulas.05,Morris2005}.

The solution $X$ to \eqref{eq.Riccati} is usually dense and it cannot be stored in case of large scale problems.
Under certain assumptions on the coefficient matrices, the singular values of the solution present a very fast decay and the symmetric positive semidefinite matrix 
$X$ can thus be well-approximated by a low rank matrix $SS^T\approx X$, $S\in\RR^{n\times t}$, $t\ll n$, so that only the low-rank factor $S$ needs to be computed and stored.
See, e.g., \cite{Benner2016a}.

Many efficient numerical methods for the solution of \eqref{eq.Riccati} have been developed in the last decades. For instance, the Newton-Kleinman
method and many of its variants \cite{Kleinman1968, Feitzinger2009,Benner2016,Benner2010, Benner1998}, 
projection methods \cite{Simoncini2014,Simoncini2016a, Jbilou2003,Heyouni2008/09},
subspace iteration methods \cite{Lin2015, Amodei2010, Benner2016a} and the very recent RADI \cite{Benner2018}. See also the survey article \cite{Benner2018a}. 

In this paper we focus on the inexact version of the Newton-Kleinman method where, at each iteration, 
a Lyapunov matrix equation needs to be solved.
In standard implementations, these linear equations are tackled independently from each other so that
the solution of the $(k+1)$-th Lyapunov equation does not exploit the information generated for computing the solution
of the previous ones at all.
Here we show how the solution of the $(k+1)$-th equation can actually profit from the computational efforts made to solve the first $k$ ones. In particular, all the Lyapunov equations of the Newton-Kleinman scheme can be solved by employing the same 
approximation space and this observation leads to a remarkable speed-up of the entire algorithm as the Newton-Kleinman iteration is only implicitly performed, while
maintaining the convergence properties of the original method. 

The most common approximation spaces used in the solution of matrix equations by projection are the extended Krylov subspace
\begin{align}\label{def.extended}
\mathbf{EK}_m^\square(A,C^T):=&\mbox{Range}([C^T,A^{-1}C^T,AC^T,A^{-2}C^T,\ldots,A^{m-1}C^T,A^{-m}C^T])\notag\\
=& \mathbf{K}_m^\square(A,C^T)\cup\mathbf{K}_m^\square(A^{-1},A^{-1}C^T), 
\end{align}
where $\mathbf{K}_m^\square(A,C^T)=\mbox{Range}([C^T,AC^T,\ldots,A^{m-1}C^T])$ is the (standard) block Krylov subspace,
see, e.g., \cite{Simoncini2007,Knizhnerman2011}, and the more general rational Krylov subspace
\begin{equation}\label{def.rational}
\mathbf{K}_m^\square(A,C^T,\mathbf{s}):=\mbox{Range}([C^T,(A-s_2I)^{-1}C^T,\ldots,\prod_{i=2}^m(A-s_iI)^{-1}C^T]), 
\end{equation}
where $\mathbf{s}=[s_2,\ldots,s_m]^T\in\CC^{m-1}$. See, e.g., \cite{Druskin2011a,Druskin2011,Druskin2014}. We thus consider
only these spaces in our analysis.

The following is a synopsis of the paper. In section~\ref{The Newton-Kleinman method} we revisit the 
Newton-Kleinman method and its inexact variant presented in \cite{Benner2016}. In section~\ref{A new iterative framework}
we show that all the iterates computed by these algorithms lie on the same subspace whose definition depends on
the choice of the initial guess $X_0$ in the Newton sequence.
In particular, in section~\ref{The extended Krylov subspace} we present the main result of the paper 
and the complete implementation of the new iterative procedure is 
 illustrated in section~\ref{Implementation details}. 
For the sake of simplicity, only the extended Krylov subspace \eqref{def.extended}
is considered in the discussion presented in section~\ref{The extended Krylov subspace}-\ref{Implementation details} but in 
section~\ref{The Rational Krylov subspace} we show how to easily adapt our new strategy when the rational Krylov subspace 
\eqref{def.rational} is adopted as approximation space.
Some novel results about the stabilizing properties of the computed solution are presented in section~\ref{Convergence results}. In particular, we show that, under certain assumptions, the matrix $A-X_kBB^T$ is stable for all $k$, where $X_k$ denotes the $k$-th iterate in the Newton sequence.
In section~\ref{Numerical examples} several numerical examples illustrate the effectiveness of the novel framework
and our conclusions are given in section~\ref{Conclusions}.

Throughout the paper we adopt the following notation.
 The matrix inner product is
defined as $\langle X, Y \rangle_F ∶= \mbox{trace}(Y^T X)$ so that the induced norm is $\|X\|_F^2= \langle X, X\rangle_F$. The Kronecker product is denoted by $\otimes$
while $I_n$ and $O_{n\times m}$ denote the identity matrix of order $n$ and the $n\times m$ zero matrix  respectively. Only one subscript is used for a square zero matrix, i.e., $O_{n\times n}=O_n$, and the subscript
is omitted whenever the dimension of $I$ and $O$ is clear from the context. Moreover, $E_i$ will denote the $i$-th block of $\ell$ columns of an identity matrix whose dimension depends on the adopted approximation space. More precisely, when the extended Krylov subspace~\eqref{def.extended} is employed, $\ell=2q$ and $E_i\in\RR^{2qm\times 2q}$ while $\ell=q$, $E_i\in\RR^{qm\times q}$, when the rational Krylov subspace~\eqref{def.rational} is selected.   
The brackets $[\cdot]$ are used to concatenate matrices of conformal dimensions. In particular, a Matlab-like
notation is adopted and $[M,N]$ denotes the matrix obtained by putting $M$ on the left of $N$ whereas $[M;N]$ the one obtained by putting $M$ on top of $N$, i.e., $[M;N]=[M^T,N^T]^T$.
The notation $\text{diag}(M,N)$ is used to denote the block diagonal matrix with diagonal blocks $M$ and $N$ and we write $A<0$ if the matrix $A$ is negative definite, i.e., if its field of values $W(A):=\{\lambda\in\mathbb{C} \text{ s.t. }
z^*(A-\lambda I)z=0,\;z\in\mathbb{C}^n,\;\|z\|_F=1\},$ $z^*$ conjugate transpose of $z$,
is contained in the open left half plane $\mathbb{C}_-$.

Denoting by $\mathcal{K}_m$ a suitable space\footnote{$\mathcal{K}_m$ as in \eqref{def.extended} or \eqref{def.rational}.}, we will always assume that a matrix $V_m\in\RR^{n\times \ell}$, $\mbox{Range}(V_m)=\mathcal{K}_m$, has orthonormal columns and it is of full rank so that $\mbox{dim}(\mathcal{K}_m)=\ell$. Indeed, if this is not the case, deflation strategies to overcome the possible linear dependence of the spanning vectors can be adopted as it is customary in block Krylov methods. See, e.g., \cite[Section 8]{Gutknecht2006}.


\section{The (inexact) Newton-Kleinman method}\label{The Newton-Kleinman method}
In this section we recall the Newton-Kleinman method and its inexact counterpart for the solution of~\eqref{eq.Riccati}. 

\begin{definition} Let $A\in\RR^{n\times n}$, $B\in\RR^{n\times p}$, and $C\in\RR^{q\times n}$. The pair $(A,B)$ is called stabilizable
 if there exists a feedback matrix $K\in\RR^{n\times p}$ such that $A-KB^T$ is stable, i.e., all the eigenvalues of $A-KB^T$ lie on the open left half complex plane $\CC_-$. The pair $(C,A)$ is called detectable if $(A^T,C^T)$ is stabilizable\footnote{Due to a different formulation of equation~\eqref{eq.Riccati}, an alternative, though mathematically equivalent, form of Definition 1 is usually adopted. See, e.g., \cite{Lancaster1995}.}
.
\end{definition}
 What follows is an assumption that will hold throughout the paper.
\begin{Ass}\label{main_assumption}
 The coefficient matrices $A\in\RR^{n\times n}$, $B\in\RR^{n\times p}$, and $C\in\RR^{q\times n}$ in \eqref{eq.Riccati}
 are such that $(A,B)$ is stabilizable and $(C,A)$ detectable.
\end{Ass}
If Assumption~\ref{main_assumption} holds, there exists a unique symmetric positive semidefinite solution $X$ to 
\eqref{eq.Riccati} which is also the unique stabilizing solution, i.e., $X$ is such that the matrix $A-XBB^T$ is stable.
 See, e.g., \cite{Lancaster1995}.

For a given $X_0$ such that $A-X_0BB^T$ is stable\footnote{Such an $X_0$ exists thanks to Assumption~\ref{main_assumption}.}, the $(k+1)$-th iteration of the Newton method is defined as 
$$\mathcal{R}'[X](X_{k+1}-X_k)=-\mathcal{R}(X_k),\quad k\geq0,$$
where $\mathcal{R}'[X]$ denotes the Fr\'{e}chet derivative of $\mathcal{R}$ at a given symmetric $X$. For the 
Riccati operator we have
$$\mathcal{R}'[X](Y)=AY+YA^T-YBB^TX-XBB^TY=(A-XBB^T)Y+Y(A-XBB^T)^T,$$
and the $(k+1)$-th iterate of the Newton-Kleinman method is thus given by the solution of the Lyapunov equation
\begin{equation}\label{Newton_step}
(A-X_kBB^T)X_{k+1}+X_{k+1}(A-X_{k}BB^T)^T=-X_{k}BB^TX_{k}-C^TC. 
\end{equation}
If the Lyapunov equations~\eqref{Newton_step} are solved exactly, the Newton-Kleinman method computes a sequence of symmetric
positive semidefinite matrices $\{X_k\}_{k\geq0}$ such that $X_k\geq X_{k+1}$ for any $k\geq 1$. Moreover, $\{X_k\}_{k\geq0}$ converges quadratically to the stabilizing solution $X$. 
See, e.g., \cite{Kleinman1968}. Furthermore, Benner and Byers showed in \cite{Benner1998}
that a line search can significantly improve 
the performance of the Newton-Kleinman method during the first Newton steps.

In our setting, due to the large problem dimensions, equations~\eqref{Newton_step} have to be iteratively solved
by one of the many efficient methods for Lyapunov equations 
present in the literature like projection methods \cite{Druskin2011, Simoncini2007},
low-rank ADI \cite{benner2008numerical, Li2004} or low-rank sign-function 
method \cite{Baur2008, Baur2006}. See also \cite{Simoncini2016} and the references therein.

The iterative solution of \eqref{Newton_step} introduces some inexactness in the Newton scheme, 
 and this leads to the so-called inexact Newton-Kleinman method whose
convergence -- under some suitable assumptions -- has been proved in \cite{Feitzinger2009}. However, the conditions considered in \cite{Feitzinger2009} 
seem difficult to meet in practice and in \cite{Benner2016} the authors showed 
that a specific line search guarantees the convergence of the inexact Newton-Kleinman scheme.

 The inexact Newton-Kleinman method with line search reads as follows.
Given a symmetric $X_k$, $\alpha>0$, $\eta_k\in(0,1)$, we want to compute a matrix $Z_{k}$ such that
\begin{equation}\label{residual_step_line_search}
\|\mathcal{R}'[X_k](Z_k)+\mathcal{R}(X_k)\|_F\leq\eta_k\|\mathcal{R}(X_k)\|_F. 
\end{equation}
The $(k+1)$-th iterate is then defined as 
\begin{equation}\label{Newton_step_linesearch}
 X_{k+1}=X_k+\lambda_kZ_k,
\end{equation}
 where the step size $\lambda_k>0$ is such that 
\begin{equation}\label{decrease_cond}
\|\mathcal{R}(X_k+\lambda_kZ_k)\|_F\leq(1-\lambda_k\alpha)\|\mathcal{R}(X_k)\|_F, 
\end{equation}
without $\lambda_k$ being too small.

 We thus need a procedure to compute the update $Z_k$ and we can proceed as follows. If $L_{k+1}:=\mathcal{R}'[X_k](Z_k)+\mathcal{R}(X_k)$, equation \eqref{residual_step_line_search} is equivalent to $\|L_{k+1}\|_F\leq\eta_k\|\mathcal{R}(X_k)\|_F$ and we can write
$$L_{k+1}=(A-X_kBB^T)(X_k+Z_k)+(X_k+Z_k)(A-X_kBB^T)^T+X_kBB^TX_k+C^TC.$$
This means that the matrix $\widetilde X_{k+1}:=X_k+Z_k$ is the solution of the Lyapunov equation
\begin{equation}\label{Newton_step_inexact}
 (A-X_kBB^T)\widetilde X_{k+1}+\widetilde X_{k+1}(A-X_{k}BB^T)^T=-X_{k}BB^TX_{k}-C^TC+L_{k+1}. 
\end{equation}
Clearly the matrix $L_{k+1}$ is never computed explicitly and the notation in \eqref{Newton_step_inexact} is used only to indicate that $\widetilde X_{k+1}$ is an inexact solution to equation \eqref{Newton_step} such that the residual norm $\|L_{k+1}\|_F$ satisfies \eqref{residual_step_line_search}. Once $\widetilde X_{k+1}$ is computed, we recover $Z_k$ by $Z_k=\widetilde X_{k+1}-X_k$.
 
We now want to compute a step-size $\lambda_k>0$ such that \eqref{decrease_cond} holds and define $X_{k+1}$ as in \eqref{Newton_step_linesearch}. 
The Riccati residual at $X_{k+1}=X_k+\lambda_kZ_k$ can be expressed as 
$$\mathcal{R}(X_k+\lambda_kZ_k)=(1-\lambda_k)\mathcal{R}(X_k)+\lambda_kL_{k+1}-\lambda_k^2Z_kBB^TZ_k,$$
so that, if $\eta_k\leq\widebar \eta<1$ and $\alpha\in (0,1-\widebar \eta)$, the sufficient decrease condition \eqref{decrease_cond} is satisfied for every $\lambda_k\in(0,(1-\alpha-\widebar \eta)\cdot\frac{\|\mathcal{R}(X_k)\|_F}{\|Z_kBB^TZ_k\|_F}]$.

In \cite{Benner2016}, two choices for the forcing parameter $\eta_k$ and for the actual computation 
of the step size $\lambda_k$ are proposed and in \cite[Theorem 10]{Benner2016} the authors showed the convergence of the inexact iterative scheme.

The Newton-Kleinman method can be formulated in different ways. For instance,  in the exact setting, if $X_1$ satisfies the equation $(A-X_0BB^T) X_1+ X_1(A-X_0BB^T)^T+C^TC=0$, then a matrix $\delta X_k$ given by the solution of the Lyapunov equation
\begin{equation}\label{reformulation}
 (A-X_kBB^T)\delta X_k+ \delta X_k(A-X_kBB^T)^T=\delta X_{k-1}BB^T \delta X_{k-1},\quad k\geq 1,\quad \delta X_0=X_1-X_0,
\end{equation}
can be computed; the next iterate of the Newton-Kleinman scheme is then defined as $X_{k+1}=X_k+\delta X_k$.
See, e.g., \cite{Banks1991}. In the recent literature, this reformulation has been shown to be very appealing when $C^TC$ is supposed to be a \emph{hierarchical} matrix and not simply low-rank. See \cite{Kressner2019}.
In our problem setting, the solution of \eqref{reformulation} may be computationally advantageous if $p$ is significantly smaller than $q$. However, 
we prefer to deal with equations of the form \eqref{Newton_step} as suggested in \cite{Feitzinger2009,Benner2016}.

In Algorithm~\ref{iNK_algorithm} the inexact Newton-Kleinman method is summarized.
\begin{algorithm}
\caption{Inexact Newton-Kleinman method with line search.\label{iNK_algorithm}}
\SetKwInOut{Input}{input}\SetKwInOut{Output}{output}
\Input{$A\in\mathbb{R}^{n\times n},$ $B\in\RR^{n\times p}$, $C\in\RR^{q\times n}$, 
$X_0\in\mathbb{R}^{n\times n}$, s.t. $A-X_0BB^T$ is stable, $\epsilon>0$, $\widebar \eta\in(0,1)$, $\alpha\in(0,1-\widebar \eta).$}
\Output{$X_k\in\mathbb{R}^{n\times n}$ approximate solution to \eqref{eq.Riccati}.}
\BlankLine
\For{$k = 0,1,\dots,$ till convergence}{
 \If{$\|\mathcal{R}(X_k)\|_F<\epsilon\cdot\|\mathcal{R}(X_0)\|_F$}{ 
\nl \textbf{Stop} and return $X_k$}
\nl Select $\eta_k\in(0,\widebar \eta)$\\
\nl Compute $\widetilde X_{k+1}$ s.t. 
$$(A-X_kBB^T)\widetilde X_{k+1}+\widetilde X_{k+1}(A-X_{k}BB^T)^T=-X_{k}BB^TX_{k}-C^TC+L_{k+1}$$ 
where $\|L_{k+1}\|_F\leq\eta_k\|\mathcal{R}(X_k)\|_F$ \label{Alg.Lyap.solves}\\
\nl Set $Z_k=\widetilde X_{k+1}-X_k$\\
\nl Compute $\lambda_k>0$ s.t. $\|\mathcal{R}(X_k+\lambda_kZ_k)\|_F\leq(1-\lambda_k\alpha)\|\mathcal{R}(X_k)\|_F$\\
\nl Set $X_{k+1}=X_k+\lambda_kZ_k$
}
\end{algorithm}

The effectiveness of Algorithm~\ref{iNK_algorithm} is strictly related to the efficiency of the Lyapunov solves 
in line~\ref{Alg.Lyap.solves}. In \cite{Benner2016}, the low-rank ADI method is employed for solving the equations in
\eqref{Newton_step} whereas in \cite{Simoncini2014} some numerical results are reported where the extended Krylov subspace method
(EKSM) is used as inner solver. 
See \cite{Simoncini2007} for an implementation of EKSM called K-PIK.
In both cases the inexact Newton-Kleinman method is not 
competitive when compared to other methods as the Newton-Kleinman method with Galerkin acceleration \cite{Benner2010}, 
projection methods \cite{Simoncini2014} and the very recent RADI \cite{Benner2018}. See, e.g., \cite{Benner2018a}.
As already mentioned, the main disadvantage of the inexact Newton-Kleinman scheme is that, at each iteration $k+1$, equation
\eqref{Newton_step} is solved independently from the previous ones. For instance, 
if a projection method like EKSM is used as Lyapunov solver,
a new subspace has to be computed from scratch at each Newton iteration. 
However, in the next section, we show that 
all the iterates computed by the Newton-Kleinman scheme \eqref{Newton_step} lie on the same space. This means 
that only one space needs to be constructed and the outer iteration, i.e., the Newton-Kleinman iteration, can be performed implicitly leading to remarkable reductions in the computational efforts.


\section{A new iterative framework}\label{A new iterative framework}

In this section we show that all the iterates computed by the Newton-Kleinman method lie in the same subspace if 
a projection method is employed in the solution of the Lyapunov equations \eqref{Newton_step}. 

Given a Lyapunov equation $AW+WA^T+C^TC=0$ where $C\in\RR^{q\times n}$ is low-rank,
projection methods compute an approximate solution of the form $W_m=V_mY_mV_m^T$ where the orthonormal 
columns of $V_m$ are a basis of a suitable subspace $\mathcal{K}_m$, i.e., 
$\mathcal{K}_m=\mbox{Range}(V_m)$, and $Y_m$ is a small square matrix computed, e.g., by imposing a Galerkin condition 
on the residual matrix $AW_m+W_mA^T+C^TC$. In Algorithm~\ref{Alg:proj} the general framework of projection methods for Lyapunov 
equations is reported. See also \cite{Simoncini2016} for more details.
\setcounter{AlgoLine}{0}
\begin{algorithm}
\caption{Galerkin projection method for the Lyapunov matrix equation.\label{Alg:proj}}
\SetKwInOut{Input}{input}\SetKwInOut{Output}{output}
\Input{$A\in\mathbb{R}^{n\times n},$ $A$ negative definite, $C\in\mathbb{R}^{q\times n}$, $\epsilon>0$}
\Output{$S_m\in\mathbb{R}^{n\times t}$, $t\leq \mbox{dim}(\mathcal{K}_m)$, $S_mS_m^T=W_m\approx W$}
\BlankLine 
\nl Set $\beta=\|CC^T\|_F$ \\
\nl Perform economy-size QR of $C^T$, $C^T=V_1 {\pmb \gamma}$\\ 
\For{$m=1, 2,\dots,$ till convergence,}{
\nl Compute next basis block $\mathcal{V}_{m+1}$ and set $V_{m+1}=[V_{m},\mathcal{V}_{m+1}]$ \\ \label{Alg.line:basis}
\nl Update $T_m=V_m^TAV_m$ \\
\nl Solve $T_mY_m+Y_mT_m^T+E_1\pmb{\gamma\gamma}^TE_1^T=0$, $\quad E_1\in\mathbb{R}^{\text{dim}(\mathcal{K}_m)\times \ell}$\\ \label{Alg.line:solve}
\nl Compute $\|R_m\|_F=\|A(V_mY_mV_m^T)+(V_mY_mV_m^T)A^T+C^TC\|_F$\\ \label{Alg.line:res}
\nl \If{$\|R_m\|_F/\beta<\epsilon$\label{Alg.line:stop}}{
\nl \textbf{Break} and go to \textbf{9} } 
  }
  \nl Compute the eigendecomposition of $Y_m$ and retain $\widehat Y\in\mathbb{R}^{\text{dim}(\mathcal{K}_m)\times t}$, $t\leq \mbox{dim}(\mathcal{K}_m)$ 
  \label{Alg.line:yhat}\\
  \nl Set $S_m=V_m\widehat Y$ \label{Alg.line:last}
\end{algorithm}

To ensure the solvability of the projected problems in line~\ref{Alg.line:solve} of Algorithm~\ref{Alg:proj}, the matrix 
$A$ is usually supposed to be negative definite as this is a sufficient condition for having a stable $T_m=V_m^TAV_m$.
However, also the stability of $T_m$ is only a sufficient condition for the well-posedness of the projected equation
in line~\ref{Alg.line:solve} and projection methods work in practice even with a coefficient matrix $A$ that is stable but not 
necessarily negative definite. See, e.g., \cite[Section 5.2.1]{Simoncini2016}.
For the Lyapunov equations~\eqref{Newton_step}, it is easy to show that the coefficient matrices $A-X_kBB^T$ are stable for all $k\geq 0$ when the exact Newton-Kleinman method is employed. See, e.g., \cite{Kleinman1968}. 
However, this is no longer straightforward in the inexact setting and in \cite[Theorem 10]{Benner2016} the authors have to assume that $A-X_{k}BB^T$ is stable for $k>k_0$. Nevertheless, in section~\ref{Convergence results} we show that the $k$-th iterate $X_k$ produced by our scheme is such that the matrix $A-X_kBB^T$ is stable if $A<0$, $\|L_{k+1}\|_F$ is sufficiently small, and the matrices $BB^T$, $C^TC$ fulfill certain conditions. Moreover, the sequence
$\{X_k\}_{k\geq 0}$ computed by our novel scheme is well-defined.


In line~\ref{Alg.line:yhat}, the matrix $\widehat Y$ denotes a 
low-rank factor of $Y_m$. If $Y_m$ is not numerically low-rank, $\widehat Y$ amounts to its Cholesky factor and $t=\text{dim}(\mathcal{K}_m)$.

The performance of projection methods mainly depends on the quality of the approximation space $\mathcal{K}_m$ employed.
As already mentioned, one of the most popular choices is the block extended Krylov subspace~\eqref{def.extended}
which leads to EKSM presented in \cite{Simoncini2007} for the solution of large-scale Lyapunov equations.

If EKSM is employed as inner solver in the Newton-Kleinman method, at each Newton step \eqref{Newton_step} 
we have to build 
a new extended Krylov subspace $\mathbf{EK}_m^\square(A-X_kBB^T,[C^T,X_kB])$ and this is not feasible in practice.
However, in the next section we show that all the iterates $X_{k+1}$ in \eqref{Newton_step} belong to the same
extended Krylov subspace whose definition depends on the choice of the initial guess $X_0$. 
This means that only one approximation space needs to be constructed to compute all the necessary iterates of the Newton-Kleinman
method.
We first show this result supposing that the $(k+1)$-th iterate $X_{k+1}$ is defined as the solution of equation~\eqref{Newton_step}, i.e., we do not perform any line search.
Then we generalize the result to the case of the solves~\eqref{Newton_step_inexact} equipped with a line search. 

Notice that the level of accuracy in the Lyapunov solves, namely $\|L_{k+1}\|_F$, does not play a role in what follows and, in principle, it could even fall short of fulfilling the condition in~\eqref{residual_step_line_search}. However, to have a meaningful sequence $\{X_k\}_{k\geq 0}$ that converges to the unique stabilizing solution $X$ to \eqref{eq.Riccati}, both $\|L_{k+1}\|_F$ and the line search must meet certain conditions. See, e.g., \cite{Benner2016} and section 4.



\subsection{The extended Krylov subspace}\label{The extended Krylov subspace}
If $A$ in \eqref{eq.Riccati} is negative definite, we can choose the initial guess of the Newton-Kleinman method to 
be zero\footnote{In general, to have a well-defined Newton sequence $\{X_k\}_{k\geq0}$ for $X_0=O$ is sufficient to have a stable $A$.
Here we suppose $A<0$ in order to apply a projection method in the solution of the first Lyapunov equation.}.
This means that the first equation to be solved in \eqref{Newton_step} is 
$$
AX_1+X_1A^T=-C^TC,
$$
and the extended Krylov subspace $\mathbf{EK}_{m_1}^\square(A,C^T)$ is constructed. 

The solution computed by EKSM is of the form $X_1=V_{m_1}Y_{m_1}V_{m_1}^T,$ $V_{m_1}\in\RR^{n\times 2qm_1}$, $\mbox{Range}(V_{m_1})=\mathbf{EK}_{m_1}^\square(A,C^T)$, $Y_{m_1}\in\RR^{2qm_1\times 2qm_1}$.
The second equation 
$(A-X_1BB^T)X_2+X_2(A-X_1BB^T)^T=-X_1BB^TX_1-C^TC$ can thus be rewritten as
$$(A-V_{m_1}\Theta_1B^T)X_2+X_2(A-V_{m_1}\Theta_1B^T)^T
=-(V_{m_1}\Theta_1)(V_{m_1}\Theta_1)^T-C^TC,\quad \Theta_1:=Y_{m_1}V_{m_1}^TB\in\RR^{2qm_1\times p}.$$
Therefore, the second space to be constructed is
$\mathbf{EK}_{m_2}^\square(A-V_{m_1}\Theta_1B^T,[C^T,V_{m_1}\Theta_1])$.

In the following theorem we show that $\mathbf{EK}_{m_2}^\square(A-V_{m_1}\Theta_1B^T,[C^T,V_{m_1}\Theta_1])$
is a subspace of $\mathbf{EK}_{\widebar m_2}^\square(A,C^T)$ for a sufficiently large $\widebar m_2$ and this happens also for
the spaces related to all the other equations of the Newton-Kleinman scheme.
\begin{theorem}\label{Theorem1}
 Let $X_{k+1}=S_{m_{k+1}}S_{m_{k+1}}^T$ be the 
 $(k+1)$-th iterate of the Newton-Kleinman scheme computed by solving \eqref{Newton_step} by EKSM.
 Suppose that also all the 
 previous Lyapunov equations of the Newton-Kleinman scheme have been solved by means of EKSM as well.
 Then 
 $$\mbox{Range}(S_{m_{k+1}})\subseteq\mathbf{EK}_{\widebar m_{k+1}}^\square(A,C^T),$$
 for a sufficiently large $\widebar m_{k+1}$ and $\widebar m_{k+1}\leq \sum_{j=1}^{k+1}m_j+2$.
\end{theorem}
\begin{proof}
 We are going to prove the statement by induction on $k$.

 For $k=0$, the equation $AX_1+X_1A^T=-C^TC$ needs to be solved. By applying EKSM
 we obtain a solution of the form $X_1=S_1S_1^T$ such that 
 $\mbox{Range}(S_1)\subseteq\mathbf{EK}_{m_1}^\square(A,C^T)$, and we can set $\widebar m_1=m_1\leq m_1+2$.
 
 We now assume that $\mbox{Range}(S_{k})\subseteq\mathbf{EK}_{\widebar m_{k}}^\square(A,C^T)$, $\widebar m_{k}\leq \sum_{j=1}^{k}m_j+2$, for a certain $k>0$ 
 and we show the statement for $k+1$. If $X_k=S_kS_k^T=V_{\widebar m_k}Y_{\widebar m_k}V_{\widebar m_k}^T$
 where $V_{\widebar m_k}\in\RR^{n\times 2q\widebar m_k}$ denotes the orthonormal basis of 
 $\mathbf{EK}_{\widebar m_{k}}^\square(A,C^T)$, the Lyapunov equation which defines $X_{k+1}$ in \eqref{Newton_step} can be 
 written as
 \begin{equation}\label{eq1_proof}
(A-V_{\widebar m_{k}}\Theta_kB^T)X_{k+1}+X_{k+1}(A-V_{\widebar m_{k}}\Theta_kB^T)^T
=-(V_{\widebar m_{k}}\Theta_k)(V_{\widebar m_{k}}\Theta_k)^T-C^TC,  
 \end{equation}
where $\Theta_k:=Y_{\widebar m_{k}}V_{\widebar m_{k}}^TB\in\RR^{2q\widebar m_{k}\times p}$. Then, the extended Krylov 
subspace to solve \eqref{eq1_proof} has the form 
$\mathbf{EK}_{m_{k+1}}^\square(A-V_{\widebar m_{k}}\Theta_kB^T,[C^T,V_{\widebar m_{k}}\Theta_1])$ and we show that this is 
a subspace of $\mathbf{EK}_{\widebar m_{k+1}}^\square(A,C^T)$ for a sufficiently large $\widebar m_{k+1}$ with $\widebar m_{k+1}\leq \sum_{j=1}^{k+1}m_j+2$.

A matrix $H\in\mathbb{R}^{n\times (p+q)}$ such that $\mbox{Range}(H)\subseteq\mathbf{EK}_{m_{k+1}}^\square(A-V_{\widebar m_{k}}\Theta_kB^T,[C^T,V_{\widebar m_{k}}\Theta_k])$ can 
be written as
$$H=\sum_{j=0}^{m_{k+1}-1}(A-V_{\widebar m_{k}}\Theta_kB^T)^j[C^T,V_{\widebar m_{k}}\Theta_k]\xi_j+
\sum_{j=1}^{m_{k+1}}(A-V_{\widebar m_{k}}\Theta_kB^T)^{-j}[C^T,V_{\widebar m_{k}}\Theta_k]\nu_j,
$$
where $\xi_j,\nu_j\in\RR^{(q+p)\times(q+p)}$.

We first focus on the polynomial part and show that 
$$\mbox{Range}\left(\sum_{j=0}^{m_{k+1}-1}(A-V_{\widebar m_{k}}\Theta_kB^T)^j[C^T,V_{\widebar m_{k}}\Theta_k]\xi_j\right)\subseteq
\mathbf{EK}_{\widebar m_{k+1}}^\square(A,C^T),$$
by induction on $j$.
For $j=0$ we have $[C^T,V_{\widebar m_{k}}\Theta_k]\xi_0$ whose range is clearly a subset of 
$\mathbf{EK}_{\widebar m_{k}}^\square(A,C^T)$ as $\mbox{Range}(C^T)\subseteq\mathbf{EK}_{\widebar m_{k}}^\square(A,C^T)$
by definition and $V_{\widebar m_{k}}$ is a basis of $\mathbf{EK}_{\widebar m_{k}}^\square(A,C^T)$.

We now assume that 
\begin{equation}\label{eq1a_proof}
\mbox{Range}((A-V_{\widebar m_{k}}\Theta_kB^T)^j[C^T,V_{\widebar m_{k}}\Theta_k]\xi_j)\subseteq 
\mathbf{EK}_{\widebar m_{k}+j}^\square(A,C^T),\;\mbox{for a certain $j>0$}, 
\end{equation}
 and we show that 
$\mbox{Range}((A-V_{\widebar m_{k}}\Theta_kB^T)^{j+1}[C^T,V_{\widebar m_{k}}\Theta_k]\xi_{j+1})\subseteq 
\mathbf{EK}_{\widebar m_{k}+j+1}^\square(A,C^T)$. We have
\begin{equation}\label{eq2_proof}
\begin{split}
(A-V_{\widebar m_{k}}\Theta_kB^T)^{j+1}[C^T,V_{\widebar m_{k}}\Theta_k]\xi_{j+1}&=
   (A-V_{\widebar m_{k}}\Theta_kB^T)(A-V_{\widebar m_{k}}\Theta_kB^T)^{j}[C^T,V_{\widebar m_{k}}\Theta_k]\xi_{j+1}\\
  &= A(A-V_{\widebar m_{k}}\Theta_kB^T)^{j}[C^T,V_{\widebar m_{k}}\Theta_k]\xi_{j+1}\\
  &\quad- V_{\widebar m_{k}}\left(\Theta_kB^T(A-V_{\widebar m_{k}}\Theta_kB^T)^{j}[C^T,V_{\widebar m_{k}}\Theta_k]\xi_{j+1}\right).
\end{split} 
\end{equation}
Since \eqref{eq1a_proof} holds,
$$\mbox{Range}(A(A-V_{\widebar m_{k}}\Theta_kB^T)^j[C^T,V_{\widebar m_{k}}\Theta_k]\xi_j)\subseteq 
A\cdot\mathbf{EK}_{\widebar m_{k}+j}^\square(A,C^T)\subseteq\mathbf{EK}_{\widebar m_{k}+j+1}^\square(A,C^T).$$
Moreover,
the second term in the right-hand side of \eqref{eq2_proof} is just a linear combination of the columns of 
$V_{\widebar m_{k}}$ and its range is thus contained in $\mathbf{EK}_{\widebar m_{k}+j}^\square(A,C^T)$.
Therefore, 
$$\mbox{Range}\left(\sum_{j=0}^{m_{k+1}-1}(A-V_{\widebar m_{k}}\Theta_kB^T)^j[C^T,V_{\widebar m_{k}}\Theta_k]\xi_j\right)\subseteq
\mathbf{EK}_{\widebar m_{k}+m_{k+1}-1}^\square(A,C^T).$$

The exact same arguments together with the 
Sherman-Morrison-Woodbury (SMW) formula \cite[Equation (2.1.4)]{Golub2013},
\begin{equation}\label{SMW}
(A-V_{\widebar m_{k}}\Theta_kB^T)^{-1}=
A^{-1}+A^{-1}V_{\widebar m_k}\Theta_k(I-B^TA^{-1}V_{\widebar m_k}\Theta_k)^{-1}B^TA^{-1}
=A^{-1}(A+V_{\widebar m_k}\Upsilon_k)A^{-1},
\end{equation}
where $\Upsilon_k:=\Theta_k(I-B^TA^{-1}V_{\widebar m_k}\Theta_k)^{-1}B^T\in\RR^{2q\widebar m_k\times n}$, show that 
$$\mbox{Range}\left(\sum_{j=1}^{m_{k+1}}(A+V_{\widebar m_k}\Theta_kB^T)^{-j}[C^T,V_{\widebar m_{k}}\Theta_k]\nu_j\right)
\subseteq\mathbf{EK}_{\widebar m_{k}+m_{k+1}+2}^\square(A,C^T).$$

Since $\mathbf{EK}_{\widebar m_{k}+m_{k+1}-1}^\square(A,C^T)\subseteq\mathbf{EK}_{\widebar m_{k}+m_{k+1}+2}^\square(A,C^T),$
we can define $\widebar m_{k+1}=\widebar m_k+m_{k+1}+2$ and get the result.
\begin{flushright}
 $\square$
\end{flushright}

\end{proof}

Theorem~\ref{Theorem1} says that we can use the same space $\mathbf{EK}_{ m}^\square(A,C^T)$
for solving all the Lyapunov equations of the Newton-Kleinman scheme \eqref{Newton_step}. Moreover, for 
 the $(k+1)$-th Lyapunov equation we do not have to recompute the space from scratch but we can 
reuse the space already computed for the previous equations and just keep expanding it.

The idea of embedding approximation spaces related to different problems in one single, possibly larger, space is not new. For instance, in the context of the solution of shifted linear systems, it has been shown in \cite{Simoncini2003} how the total number of iteration of (restarted) FOM applied to a sequence of shifted linear systems simultaneously is equal to the one achieved by applying (restarted) FOM to the single linear system with the slowest convergence rate. Similarly, in our context, the
dimension of the space constructed to solve all the Lyapunov equations of the Newton-Kleinman scheme is equal to the 
dimension of $\mathbf{EK}_m^\square(A,C^T)$ when this is adopted as approximation space for the solution of the last equation in the sequence. The \emph{feedback invariant} property of $\mathbf{EK}_m^\square(A,C^T)$ may be reminiscent of the shifted invariant feature of the Krylov subspace exploited in the solution of sequences of shifted linear systems. However, all the linear systems are solved at the same time in the latter problem setting while we have to solve the Lyapunov equations in \eqref{Newton_step} serially as the $(k+1)$-th equation depends on the solution of the previous one.

A result similar to the one stated in Theorem~\ref{Theorem1} can be shown also in the case of the solves~\eqref{Newton_step_inexact} equipped with the line search
\eqref{Newton_step_linesearch}.
\begin{Cor} \label{Cor1}
 Let $\widetilde X_{k+1}=\widetilde S_{m_{k+1}}\widetilde S_{m_{k+1}}^T$ be the solution 
 to \eqref{Newton_step_inexact} computed by EKSM.
 Suppose that also all the 
 previous Lyapunov equations of the Newton-Kleinman scheme have been solved by means of EKSM as well.
 Then $X_{k+1}=X_k+\lambda_kZ_k=P_{k+1}P_{k+1}^T$, $P_{k+1}$ low-rank, is such that
 $$\mbox{Range}(P_{k+1})\subseteq\mathbf{EK}_{\widebar m_{k+1}}^\square(A,C^T),$$
 for a sufficiently large $\widebar m_{k+1}$, 
 $\widebar m_{k+1}\leq \sum_{j=1}^{k+1}m_j+2$.
\end{Cor}
\begin{proof}
 We again prove the statement by induction on $k$. 
 For $k=1$, we compute the matrix $\widetilde X_1=\widetilde S_{m_1}\widetilde S_{m_1}^T$,
 $\mbox{Range}(\widetilde S_{m_1})\subseteq\mathbf{EK}_{m_1}^\square(A,C^T)$, which is an approximate solution of the equation 
 $AX+XA^T=-C^TC$. Then, since $X_0=O$, we define the first iterate of the Newton sequence as $X_1=\lambda_1\widetilde X_1$
 so that $X_1$ can be written as $X_1=P_1P_1^T$, $P_1=\sqrt{\lambda_1}\widetilde S_{m_1}$ and 
 $\mbox{Range}(P_{1})\subseteq\mathbf{EK}_{m_{1}}^\square(A,C^T);$ $\widebar m_1 =m_1\leq m_1+2$.
 
 We now suppose that the statement has been proven for a certain $k>1$ and we show it for $k+1$.
 Let $\widetilde X_{k+1}=\widetilde S_{m_{k+1}}\widetilde S_{m_{k+1}}^T$ be the approximate solution of the equation
 $(A-X_kBB^T)X+X(A-X_kBB^T)^T=-X_kBB^TX_k-C^TC$ computed by the EKSM. Since $X_k=P_kP_k^T$ is such that 
 $\mbox{Range}(P_k)\subseteq \mathbf{EK}_{\widebar m_{k}}^\square(A,C^T)$, $\widebar m_{k}\leq \sum_{j=1}^{k}m_j+2$, by inductive hypothesis,
 with the same argument of Theorem~\ref{Theorem1} we can show that 
 $\mbox{Range}(\widetilde S_{m_{k+1}})\subseteq \mathbf{EK}_{\widebar m_{k+1}}^\square(A,C^T)$ for a sufficiently 
 large $\widebar m_{k+1}$, $\widebar m_{k+1}\leq \sum_{j=1}^{k+1}m_j+2$. 
 Then, following section~\ref{The Newton-Kleinman method} we define $Z_k=\widetilde X_{k+1}-X_k$, 
 so that the $(k+1)$-th iterate of the Newton sequence is
 \begin{equation*}
 \begin{split}
 X_{k+1}&=X_k+\lambda_kZ_k=(1-\lambda_k)X_k+\lambda_k\widetilde X_{k+1}=
 (1-\lambda_k)P_kP_k^T+\lambda_k\widetilde S_{m_{k+1}}\widetilde S_{m_{k+1}}^T \\
 &=  [(1-\lambda_k)P_k,\lambda_k \widetilde S_{m_{k+1}}][P_k,\widetilde S_{m_{k+1}}]^T
 =P_{k+1}P_{k+1}^T,\\
 \end{split}  
 \end{equation*}
 where $P_{k+1}:=[\sqrt{1-\lambda_k}P_k,\sqrt{\lambda_k} \widetilde S_{m_{k+1}}]$ is such that
 $\mbox{Range}(P_{k+1})\subseteq \mathbf{EK}_{\widebar m_{k+1}}^\square(A,C^T)$, $\widebar m_{k+1}\leq \sum_{j=1}^{k+1}m_j+2$, as
 $\mbox{Range}(P_k)\subseteq \mathbf{EK}_{\widebar m_{k}}^\square(A,C^T)\subseteq \mathbf{EK}_{\widebar m_{k+1}}^\square(A,C^T)$.
 \begin{flushright}
 $\square$
\end{flushright}

\end{proof}

The estimate on the number of iterations $\widebar m_{k+1}$
given in Theorem~\ref{Theorem1} and Corollary~\ref{Cor1} is very rough and it is provided only for showing that the dimension of $\mathbf{EK}_{\widebar m_{k+1}}^\square(A,C^T)$ is bounded by a constant which is smaller than $n$ if $m_j$ is moderate for $j=1,\ldots,k+1$.

For a given tolerance $\epsilon$, the actual dimension of $\mathbf{EK}_{\widebar m_{k+1}}^\square(A,C^T)$ to achieve $\|(A-X_kBB^T)X_{k+1}+X_{k+1}(A-X_kBB^T)^T+X_kBB^TX_k+C^TC \|_F\leq\epsilon$ is in general much smaller than
$2q\left(\sum_{j=1}^{k+1}m_j+2\right)$. See, e.g., Example~\ref{Ex.1}.

\subsection{Implementation details}\label{Implementation details}
In this section we present how to fully exploit Theorem~\ref{Theorem1} and Corollary~\ref{Cor1} by merging the (inexact)
Newton-Kleinman method in a projection procedure.
Also here we present our strategy by first assuming 
that the line search is not performed  
and then we generalize the approach.

We start by solving the equation $AX_1+X_1A^T=-C^TC$ by projection onto the extended Krylov subspace 
$\mathbf{EK}_{m_1}^\square(A,C^T)=\mbox{Range}(V_{m_1})$.
As outlined in Algorithm~\ref{Alg:proj}, if $T_{m_1}:=V_{m_1}^TAV_{m_1}$, at each iteration of EKSM we have to solve
the projected equation
\begin{equation}\label{projected_eq1}
T_{m_1}Y+YT_{m_1}^T+E_1\pmb{\gamma\gamma}^TE_1^T=0, 
\end{equation}
where $E_1\in\mathbb{R}^{2qm_1\times 2q}$ corresponds to the first $2q$ columns of $I_{2qm_1}$,  and $\pmb{\gamma}\in\RR^{2q \times q}$, $C^T=V_1\pmb{\gamma}$. 
Since equation \eqref{projected_eq1} is of small dimension, decomposition based methods as the Bartels-Stewart 
method \cite{Bartels1972} or the Hammarling method \cite{Hammarling1982} can be employed for its solution.

If at iteration $\widebar m_1$ the Lyapunov residual norm 
$\|R_{\widebar m_1}\|_F=\|A(V_{\widebar m_1}Y_{\widebar m_1}V_{\widebar m_1}^T)+(V_{\widebar m_1}Y_{\widebar m_1}V_{\widebar m_1}^T)A^T+C^TC \|_F$ is sufficiently small, we define
$Y_{\widebar m_1}:=Y$ and check the residual norm of the Riccati equation, namely
$\|\mathcal{R}(V_{\widebar m_1}Y_{\widebar m_1}V_{\widebar m_1}^T)\|_F$. If this is sufficiently small we have completed
the procedure and $V_{\widebar m_1}Y_{\widebar m_1}V_{\widebar m_1}^T$ is the sought approximated solution to 
\eqref{eq.Riccati}, otherwise we pass to solve the second equation of the Newton scheme.
We can write 
\begin{equation}\label{eq.2}
(A-V_{\widebar m_1}Y_{\widebar m_1}V_{\widebar m_1}^TBB^T)X_2+X_2(A-V_{\widebar m_1}Y_{\widebar m_1}V_{\widebar m_1}^TBB^T)^T
= -V_{\widebar m_1}Y_{\widebar m_1}V_{\widebar m_1}^TBB^TV_{\widebar m_1}Y_{\widebar m_1}V_{\widebar m_1}^T
-C^TC,
\end{equation}
and Theorem~\ref{Theorem1} says that $\mathbf{EK}_{\widebar m_2}^\square(A,C^T)$, $\widebar m_2\geq \widebar m_1$,
is still a good approximation space for
solving it. We thus start by projecting \eqref{eq.2} onto the already computed space $\mathbf{EK}_{\widebar m_1}^\square(A,C^T)$
getting
\begin{equation}\label{eq.2projected}
(T_{\widebar m_1}-Y_{\widebar m_1}B_{\widebar m_1}B_{\widebar m_1}^T)Y+ 
Y(T_{\widebar m_1}-Y_{\widebar m_1}B_{\widebar m_1}B_{\widebar m_1}^T)^T=-Y_{\widebar m_1}B_{\widebar m_1}B_{\widebar m_1}^TY_{\widebar m_1}
-E_1\pmb{\gamma\gamma}^TE_1^T,
\end{equation}
where $B_{\widebar m_1}=V_{\widebar m_1}^TB$. Notice that $B_{\widebar m_1}$ 
can be computed on the fly performing $2q$ inner products per iteration. 

It may happen that $\mathbf{EK}_{\widebar m_1}^\square(A,C^T)$ is already a good approximation space for 
equation \eqref{eq.2}, that is, the Lyapunov residual norm 
$\|R_{\widebar m_2}\|_F=\|(A-V_{\widebar m_1}Y_{\widebar m_1}V_{\widebar m_1}^TBB^T)V_{\widebar m_1}YV_{\widebar m_1}^T+
V_{\widebar m_1}YV_{\widebar m_1}^T(A-V_{\widebar m_1}Y_{\widebar m_1}V_{\widebar m_1}^TBB^T)^T
+V_{\widebar m_1}Y_{\widebar m_1}V_{\widebar m_1}^TBB^TV_{\widebar m_1}Y_{\widebar m_1}V_{\widebar m_1}^T
+C^TC\|_F$,
where $Y$ is the solution to \eqref{eq.2projected},
is sufficiently small. If this is the case we set $\widebar  m_2=\widebar m_1$, $Y_{\widebar m_2}:=Y$ and check 
$\|\mathcal{R}(V_{\widebar m_2}Y_{\widebar m_2}V_{\widebar m_2}^T)\|_F$. Otherwise, we expand the space by computing the next basis block $\mathcal{V}_{\widebar m_1+1}\in\RR^{n\times 2q}$
such that $V_{\widebar m_1+1}:=[V_{\widebar m_1},\mathcal{V}_{\widebar m_1+1}]$ has orthonormal columns and
$\mbox{Range}(V_{\widebar m_1+1})=\mathbf{EK}_{\widebar m_1+1}^\square(A,C^T)$.

In the next proposition we show how to easily compute the projection of the current Lyapunov equation once the  subspace has been expanded and a cheap computation of the residual norm.
%
\begin{Prop}\label{Prop_projectedeqs}
 Let $X_k=V_{\widebar m_k}Y_{\widebar m_k}V_{\widebar m_k}^T$, 
 $\mbox{Range}(V_{\widebar m_k})=\mathbf{EK}_{\widebar m_k}^\square(A,C^T)$ be the $k$-th iterate of the Newton-Kleinman scheme.
 Then, the projection of the $(k+1)$-th equation \eqref{Newton_step}
 onto $\mathbf{EK}_{\widebar m_{k+1}}^\square(A,C^T)=\mbox{Range}(V_{\widebar m_{k+1}})$,
 $\widebar m_{k+1}\geq \widebar m_{k}$, is given by
{\small
\begin{equation}\label{eq.kprojected}
Q_{\widebar m_{k+1}}Y+ 
YQ_{\widebar m_{k+1}}^T=-\mbox{diag}(Y_{\widebar m_{k}},O_{2q(\widebar m_{k+1}-\widebar m_{k})})B_{\widebar m_{k+1}}B_{\widebar m_{k+1}}^T
\mbox{diag}(Y_{\widebar m_{k}},O_{2q(\widebar m_{k+1}-\widebar m_{k})})-E_1\pmb{\gamma\gamma}^TE_1^T,
\end{equation}
}
 where $Q_{\widebar m_{k+1}}=T_{\widebar m_{k+1}}-\mbox{diag}(Y_{\widebar m_{k}},O_{2q(\widebar m_{k+1}-\widebar m_{k})})B_{\widebar m_{k+1}}
B_{\widebar m_{k+1}}^T$, 
 $T_{\widebar m_{k+1}}=V_{\widebar m_{k+1}}^TAV_{\widebar m_{k+1}}$, $B_{\widebar m_{k+1}}=V_{\widebar m_{k+1}}^TB$
 and $C^T=V_1\pmb{\gamma}$.
  Moreover,
the solution 
$Y_{\widebar m_{k+1}}$ to \eqref{eq.kprojected} is such that 
\begin{multline}\label{res.projected}
\|(A-X_{k}BB^T)(V_{\widebar m_{k+1}}Y_{\widebar m_{k+1}}V_{\widebar m_{k+1}}^T)+
(V_{\widebar m_{k+1}}Y_{\widebar m_{k+1}}V_{\widebar m_{k+1}}^T)(A-X_{k}BB^T)^T+X_{k}BB^TX_{k}+C^TC\|_F=\\
=\sqrt{2}\|Y_{\widebar m_{k+1}}\underline{T}_{\widebar m_{k+1}}^T
E_{\widebar m_{k+1}+1}\|_F,
\end{multline}
where $\underline{T}_{\widebar m_{k+1}}=V_{\widebar m_{k+1}+1}^TAV_{\widebar m_{k+1}}$.
\end{Prop}
\begin{proof}
 We show the statements by induction on $k$.
 Since we suppose $A$ negative definite, for $k=0$ we have $X_0=O$ and denoting by 
 $Y_{\widebar m_{0}}:=V_1^TX_0V_1=O_{2q}$, $\mbox{Range}(V_1)=\mathbf{EK}_{1}^\square(A,C^T)$, we can write
 the projection of $AX+XA^T=-C^TC$ onto $\mathbf{EK}_{\widebar m_1}^\square(A,C^T)$, $m_1\geq 1$, as
\begin{equation}\label{eq1projected_proof}
Q_{\widebar m_{1}}Y+ 
YQ_{\widebar m_{1}}^T=\\
=-\mbox{diag}(Y_{\widebar m_{0}},O_{2q(\widebar m_{1}-1)})B_{\widebar m_{1}}B_{\widebar m_{1}}^T
\mbox{diag}(Y_{\widebar m_{0}},O_{2q(\widebar m_{1}-1)})-E_1\pmb{\gamma\gamma}^TE_1^T,
\end{equation}
 with $Q_{\widebar m_{1}}=T_{\widebar m_{1}}-\mbox{diag}(Y_{\widebar m_{0}},O_{2q(\widebar m_{1}-1)})B_{\widebar m_{1}}
B_{\widebar m_{1}}^T$.
Moreover,
$$
\|A(V_{\widebar m_{1}}Y_{\widebar m_{1}}V_{\widebar m_{1}}^T)+
(V_{\widebar m_{1}}Y_{\widebar m_{1}}V_{\widebar m_{1}}^T)A^T+C^TC\|_F
=\sqrt{2}\|Y_{\widebar m_{1}}\underline{T}_{\widebar m_{1}}^T
E_{\widebar m_{1}+1}\|_F,
$$
 where $Y_{\widebar m_{1}}$ denotes the solution to \eqref{eq1projected_proof}. See, e.g., \cite{Jaimoukha1994}.
 
 We now suppose that the statements hold for a certain $k>0$ and we show them for $k+1$.
 If $X_k=V_{\widebar m_k}Y_{\widebar m_k}V_{\widebar m_k}^T$, 
 $\mbox{Range}(V_{\widebar m_k})=\mathbf{EK}_{\widebar m_k}^\square(A,C^T)$, 
 then we can write the $(k+1)$-th equation of the Newton scheme as 
 \begin{equation}\label{eqk_proof}  
(A-V_{\widebar m_k}Y_{\widebar m_k}V_{\widebar m_k}^TBB^T)X+X(A-V_{\widebar m_k}Y_{\widebar m_k}V_{\widebar m_k}^TBB^T)^T
 =-V_{\widebar m_k}Y_{\widebar m_k}V_{\widebar m_k}^TBB^TV_{\widebar m_k}Y_{\widebar m_k}V_{\widebar m_k}^T-C^TC.
 \end{equation}
 If $\text{Range}(V_{\widebar m_{k+1}})=\mathbf{EK}_{\widebar m_{k+1}}^\square(A,C^T)$  for 
 $\widebar m_{k+1}\geq \widebar m_{k}$, then 
 $V_{\widebar m_{k+1}}^TV_{\widebar m_{k}}=[I_{ 2q\widebar m_{k}};O_{2q({\widebar m_{k+1}}-\widebar m_{k})\times 2qm_{k}}]$.
 Therefore, pre and post multiplying equation \eqref{eqk_proof} by 
 $V_{\widebar m_{k+1}}^T$ and $V_{\widebar m_{k+1}}$ respectively, we get 
\begin{multline*}
(T_{\widebar m_{k+1}}-[Y_{\widebar m_{k}};O_{2q({\widebar m_{k+1}}-\widebar m_{k})\times 2qm_{k}}]V_{\widebar m_k}^TB
B_{\widebar m_{k+1}}^T)Y+ 
Y(T_{\widebar m_{k+1}}-[Y_{\widebar m_{k}};O_{2q({\widebar m_{k+1}}-\widebar m_{k})\times 2qm_{k}}]V_{\widebar m_k}^TB
B_{\widebar m_{k+1}}^T)^T=\\
=-[Y_{\widebar m_{k}};O_{2q({\widebar m_{k+1}}-\widebar m_{k})\times 2qm_{k}}]V_{\widebar m_k}^TBB^TV_{\widebar m_k}
[Y_{\widebar m_{k}};O_{2q({\widebar m_{k+1}}-\widebar m_{k})\times 2qm_{k}}]^T-E_1\pmb{\gamma\gamma}^TE_1^T,
\end{multline*}
 and noticing that 
 $[Y_{\widebar m_{k}};O_{(2q{\widebar m_{k+1}}-\widebar m_{k})\times 2qm_{k}}]V_{\widebar m_k}^T= \mbox{diag}(Y_{\widebar m_{k}},O_{2q(\widebar m_{k+1}-\widebar m_{k})})V_{\widebar m_{k+1}}^T$ we have the result.
 

In conclusion, if $Y_{\widebar m_{k+1}}$ denotes the solution to \eqref{eq.kprojected},
it is easy to show that the residual norm of the Lyapunov equation
$(A-X_{k}BB^T)(V_{\widebar m_{k+1}}Y_{\widebar m_{k+1}}V_{\widebar m_{k+1}}^T)+
(V_{\widebar m_{k+1}}Y_{\widebar m_{k+1}}V_{\widebar m_{k+1}}^T)(A-X_{k}BB^T)^T=-X_{k}BB^TX_{k}-C^TC$ can be computed 
very cheaply as in \eqref{res.projected}.
 The proof follows the same line of the proof in \cite[Proposition 3.3]{Simoncini2007}  
recalling that $X_{k}=V_{\widebar m_{k}}Y_{\widebar m_{k}}V_{\widebar m_{k}}^T=
 V_{\widebar m_{k+1}}\mbox{diag}(Y_{\widebar m_{k}},O_{2q(\widebar m_{k+1}-\widebar m_{k})})V_{\widebar m_{k+1}}^T$.
 \begin{flushright}
 $\square$
\end{flushright}

 \end{proof}

A similar result can be shown also in case of the Newton-Kleinman method equipped with a line search. 
Indeed, at the $(k+1)$-th iteration, we define $X_{k+1}$ as in \eqref{Newton_step_linesearch} where
$Z_k=\widetilde X_{k+1}-X_k$  and $\widetilde X_{k+1}$ is the solution to~\eqref{Newton_step_inexact}.
Assuming $X_k=V_{\widebar m_k}Y_{\widebar m_k}V_{\widebar m_k}^T$, 
 $\mbox{Range}(V_{\widebar m_k})=\mathbf{EK}_{\widebar m_k}^\square(A,C^T)$,
  the projection of 
 $$(A-X_{k}BB^T)X+X(A-X_{k}BB^T)^T=-X_{k}BB^TX_{k}-C^TC,$$
 onto $\mathbf{EK}_{\widebar m_{k+1}}^\square(A,C^T)=\mbox{Range}(V_{\widebar m_{k+1}})$,
 $\widebar m_{k+1}\geq \widebar m_{k}$, is still of the form \eqref{eq.kprojected}.
 This means that the residual norm can be still computed as in \eqref{res.projected}.
 Once $\|L_{k+1}\|_F=\sqrt{2}\|\widetilde Y_{\widebar m_{k+1}}\underline{T}_{\widebar m_{k+1}}^T
E_{\widebar m_{k+1}+1}\|_F\leq \eta_k\|\mathcal{R}(X_k)\|_F$, where $\widetilde Y_{\widebar m_{k+1}}$ is the solution 
to \eqref{eq.kprojected},
we define $\widetilde X_{k+1}:=V_{\widebar m_{k+1}}
\widetilde Y_{\widebar m_{k+1}} V_{\widebar m_{k+1}}^T$, and we have
$$\begin{array}{rll}
X_{k+1}&=&X_k+\lambda_kZ_k=(1-\lambda_k)X_k+\lambda_k\widetilde X_{k+1}\\
&=&(1-\lambda_k)V_{\widebar m_k}Y_{\widebar m_k}V_{\widebar m_k}^T+\lambda_kV_{\widebar m_{k+1}}
\widetilde Y_{\widebar m_{k+1}} V_{\widebar m_{k+1}}^T\\
&=&(1-\lambda_k)V_{\widebar m_{k+1}}\mbox{diag}(Y_{\widebar m_k},O_{2q(\widebar m_{k+1}-\widebar m_{k})})V_{\widebar m_{k+1}}^T+\lambda_kV_{\widebar m_{k+1}}
\widetilde Y_{\widebar m_{k+1}} V_{\widebar m_{k+1}}^T \\
&=&V_{\widebar m_{k+1}}\left((1-\lambda_k)\cdot\mbox{diag}(Y_{\widebar m_k},O_{2q(\widebar m_{k+1}-\widebar m_{k})})
+\lambda_k\widetilde Y_{\widebar m_{k+1}} \right)V_{\widebar m_{k+1}}^T\\
&=&V_{\widebar m_{k+1}}Y_{\widebar m_{k+1}} V_{\widebar m_{k+1}}^T, 
\end{array}
$$
where $Y_{\widebar m_{k+1}}:=(1-\lambda_k)\cdot\mbox{diag}(Y_{\widebar m_k},O_{2q(\widebar m_{k+1}-\widebar m_{k})})
+\lambda_k\widetilde Y_{\widebar m_{k+1}}$.

We are thus left with showing that the line search, i.e., the computation of the $\lambda_k$'s, can be cheaply
carried out on the current subspace with no need to go back to $\RR^n$.
In \cite{Benner2016}, the authors show that $\|\mathcal{R}(X_k+\lambda Z_k)\|_F^2$ is a quartic polynomial in $\lambda$ of the form
\begin{equation}\label{polynomial_linesearch}
 p_k(\lambda)=\|\mathcal{R}(X_k+\lambda Z_k)\|_F^2=
(1-\lambda)^2\alpha_k+\lambda^2\beta_k+\lambda^4\delta_k+2\lambda(1-\lambda)\gamma_k-2\lambda^2(1-\lambda)\epsilon_k
-2\lambda^3\zeta_k,
\end{equation}
where
\begin{equation}\label{polynomial_coefficients}
 \begin{array}{ll}
  \alpha_k=\|\mathcal{R}(X_k)\|_F^2, & \delta_k=\|Z_kBB^TZ_k\|_F^2, \\
  \beta_k=\|L_{k+1}\|_F^2,&  \epsilon_k=\langle \mathcal{R}(X_k),Z_kBB^TZ_k\rangle_F,\\
  \gamma_k=\langle \mathcal{R}(X_k),L_{k+1}\rangle_F, & \zeta_k=\langle L_{k+1},Z_kBB^TZ_k\rangle_F.
 \end{array}
\end{equation}
If $\|L_{k+1}\|_F=0$, the polynomial in \eqref{polynomial_linesearch} has a local minimizer $\lambda_k\in(0,2]$ such that,
if $A-X_kBB^T$ is stable and $X_{k+1}$ is computed by using such $\lambda_k$, also $A-X_{k+1}BB^T$ is stable.
See \cite{Benner1998}. However, in \cite{Benner2016} the authors state that, in general, this no longer holds if $\|L_{k+1}\|_F\neq0$. 
 Nevertheless, we show that, in our particular framework, $p_k(\lambda)$ still has a local minimizer
 in $(0,2]$ and we can thus compute the step-size as
%
\begin{equation}\label{stepsizeselection}
\lambda_k=\argmin_{(0,2]}p_k(\lambda). 
\end{equation}
We first derive new expressions for the coefficients \eqref{polynomial_coefficients} that will help us to prove the existence of a local minimizer of $p_k(\lambda)$ in $(0,2]$.

\begin{Prop}\label{polynomial_coefficients_computation}
 The coefficients in \eqref{polynomial_coefficients} are such that
 
 \begin{align*}
  \alpha_k
    =&\left\|T_{\widebar m_k}Y_{\widebar m_k}+Y_{\widebar m_k}T_{\widebar m_k}^T-Y_{\widebar m_k}B_{\widebar m_k}B_{\widebar m_k}^TY_{\widebar m_k}+E_1\pmb{\gamma\gamma}^TE_1^T\right\|_F^2+2\left\|Y_{\widebar m_k}\underline{T}_{\widebar m_k}^TE_{\widebar m_k+1}\right\|_F^2, \\
    \\
 \beta_k=&2\left\|\widetilde Y_{\widebar m_{k+1}}\underline{T}_{\widebar m_{k+1}}E_{\widebar m_{k+1}+1}\right\|_F^2,\\
 \\
 \gamma_k=&\left\{ \begin{array}{l}
 2 \langle E_{\widebar m_k+1}^T\underline{T}_{\widebar m_k}Y_{\widebar m_k}, 
 E_{\widebar m_{k+1}+1}^T\underline{T}_{\widebar m_{k+1}}
 \widetilde Y_{\widebar m_{k+1}}
\rangle_F, \quad \mbox{if } \widebar m_{k+1}=\widebar m_{k},\\
 0, \quad \mbox{if } \widebar m_{k+1}>\widebar m_{k},
\end{array}\right.\\
\\
 \delta_k=&
\left\|\left(\widetilde Y_{\widebar m_{k+1}}-\mbox{diag}(Y_{\widebar m_{k}},O_{2q(\widebar m_{k+1}-\widebar m_{k})})\right)B_{\widebar m_{k+1}}B_{\widebar m_{k+1}}^T
\left(\widetilde Y_{\widebar m_{k+1}}
-\mbox{diag}(Y_{\widebar m_{k}},O_{2q(\widebar m_{k+1}-\widebar m_{k})})\right)\right\|^2_F,
 \end{align*}

$$
 \epsilon_k
 =\left\{
\begin{array}{l}
\begin{array}{l}
\left\langle T_{\widebar m_{k}}Y_{\widebar m_{k}}+Y_{\widebar m_{k}}T_{\widebar m_{k}}^T-Y_{\widebar m_{k}}B_{\widebar m_{k}}B_{\widebar m_{k}}^TY_{\widebar m_{k}}+E_1\pmb{\gamma\gamma}^TE_1^T,\right.\\
\left.\left(\widetilde Y_{\widebar m_{k+1}}-Y_{\widebar m_{k}}\right)B_{\widebar m_{k+1}}B_{\widebar m_{k+1}}^T\left(\widetilde Y_{\widebar m_{k+1}} - Y_{\widebar m_{k}}\right)\right\rangle_F,
\end{array} \quad \mbox{if } \widebar m_{k+1}= \widebar m_{k},\\
\\
\begin{array}{l}
\left\langle T_{\widebar m_{k}}Y_{\widebar m_{k}}
  +Y_{\widebar m_{k}}T_{\widebar m_{k}}^T
 -Y_{\widebar m_{k}}B_{\widebar m_{k}}B_{\widebar m_{k}}^TY_{\widebar m_{k}}
   +E_1\pmb{\gamma\gamma}^TE_1^T,\right.\\
   
[I_{2q\widebar m_k},O_{2q\widebar m_k\times 2q(\widebar m_{k+1}-\widebar m_k)}] \left(\widetilde Y_{\widebar m_{k+1}}-\mbox{diag}(Y_{\widebar m_{k}},O_{2q(\widebar m_{k+1}-\widebar m_{k})})\right)B_{\widebar m_{k+1}}B_{\widebar m_{k+1}}^T\left(\widetilde Y_{\widebar m_{k+1}} \right. \\
 \left.\left. -\mbox{diag}(Y_{\widebar m_{k}},O_{2q(\widebar m_{k+1}-\widebar m_{k})})\right)[I_{2q\widebar m_k};O_{2q(\widebar m_{k+1}-\widebar m_k)\times 2q\widebar m_k}]\right\rangle_F\\
 +\left\langle E_{\widebar m_k+1}E_{\widebar m_k+1}^T\underline{T}_{\widebar m_k}[Y_{\widebar m_k},O_{2q\widebar m_k\times 2q}]+[Y_{\widebar m_k};O_{2q\times 2q\widebar m_k}]\underline{T}_{\widebar m_k}^TE_{\widebar m_k+1}E_{\widebar m_k+1}^T,\right. \\

 [I_{2q(\widebar m_k+1)},O_{2q(\widebar m_k+1)\times 2q(\widebar m_{k+1}-\widebar m_k-1)}]\left(\widetilde Y_{\widebar m_{k+1}}-\mbox{diag}(Y_{\widebar m_{k}},O_{2q(\widebar m_{k+1}-\widebar m_{k})})\right)B_{\widebar m_{k+1}}
    B_{\widebar m_{k+1}}^T\left(\widetilde Y_{\widebar m_{k+1}}\right.\\
    
    \left.\left.- \mbox{diag}(Y_{\widebar m_{k}},O_{2q(\widebar m_{k+1}-\widebar m_{k})})\right)[I_{2q(\widebar m_k+1)};O_{2q(\widebar m_{k+1}-\widebar m_k-1)\times 2q(\widebar m_k+1)}]\right\rangle_F.
\end{array}\hspace{-1.2cm}
 \mbox{if } \widebar m_{k+1}> \widebar m_{k},

\end{array}\right.
$$
and
$$\zeta_k=0.$$
\end{Prop}
See the Appendix for the proof.

Proposition~\ref{polynomial_coefficients_computation} shows
how only matrices of size (at most) $2q(\widebar m_{k+1}+1)$
are actually involved in the computation of the coefficients \eqref{polynomial_coefficients} and we just  
need information that is available
in the current subspace $\mathbf{EK}_{\widebar m_{k+1}}^\square(A,C^T)$ to define  $p_k(\lambda)$ without any 
backward transformations to $\RR^n$.

By exploiting the expressions in Proposition~\ref{polynomial_coefficients_computation} we are now able to show the existence of a local minimizer $\lambda_k\in(0,2]$ of $p_k(\lambda)$.
\begin{Prop}\label{existence_lambda}
 If $\widebar m_{k+1}>\widebar m_k$, the polynomial $p_k(\lambda)$ has a local minimizer $\lambda_k\in(0,2]$.
\end{Prop}
\begin{proof}
 If $\widebar m_{k+1}>\widebar m_k$, by exploiting the 
 expressions in Proposition~\ref{polynomial_coefficients_computation}, the polynomial
 $p_k(\lambda)$ in \eqref{polynomial_linesearch} can be written as
 $$ p_k(\lambda)=
(1-\lambda)^2\alpha_k+\lambda^2\beta_k+\lambda^4\delta_k-2\lambda^2(1-\lambda)\epsilon_k,
$$
whose first derivative is
 $$ p'_k(\lambda)=
-2(1-\lambda)\alpha_k+2\lambda\beta_k+4\lambda^3\delta_k-4\lambda\epsilon_k+6\lambda^2\epsilon_k.
$$
Therefore, $p'_k(0)=-2\alpha_k<0$ as $\alpha_k=\|\mathcal{R}(X_k)\|_F^2>0$. Notice that if $\alpha_k=0$,
this means that $X_k$ is the exact solution to \eqref{eq.Riccati} and we do not need to compute any step-size $\lambda_k$.

Moreover,
$$p'_k(2)=2\alpha_k+4\beta_k+32\delta_k+16\epsilon_k\geq 2\alpha_k+32\delta_k+16\epsilon_k=2\|\mathcal{R}(X_k)+4Z_kBB^TZ_k\|_F^2,$$
as $\beta=\|L_{k+1}\|^2_F\geq 0$. Since $\|\mathcal{R}(X_k)+4Z_kBB^TZ_k\|_F^2\geq 0$,
also $p'_k(2)\geq 0$ and there exists a local minimizer $\lambda_k$ of $p_k$ in $(0,2]$.
\begin{flushright}
 $\square$
\end{flushright}

\end{proof}

In our numerical experience it is very rare to have $\widebar m_{k+1}=\widebar m_k$. Indeed, it is unlikely that the space used for solving the $k$-th Lyapunov equation in the Newton sequence contains enough spectral information to solve also the $(k+1)$-th one. This may happen for the very first couple of Lyapunov equations, i.e., it has happened that $\widebar m_2=\widebar m_1$. Since for the subsequent equations we have to expand the space anyway, we suggest performing an extra iteration when 
$\widebar m_{k+1}=\widebar m_k$, so that $\widebar m_{k+1}>\widebar m_k$, and then compute $\lambda_k$ as in \eqref{stepsizeselection}. 

Once $\lambda_k$ is computed, defining $Y_{\widebar m_{k+1}}:=(1-\lambda_k)\cdot\mbox{diag}(Y_{\widebar m_k},O_{2q(\widebar m_{k+1}-\widebar m_{k})})
+\lambda_k\widetilde Y_{\widebar m_{k+1}}$, we can cheaply 
evaluate the residual norm of the new approximate solution 
$X_{k+1}:=V_{\widebar m_{k+1}}Y_{\widebar m_{k+1}}V_{\widebar m_{k+1}}^T$ to the Riccati equation by
\begin{align}\label{Riccati_residualnorm}
 \|\mathcal{R}(X_{k+1})\|_F^2=& \|\mathcal{R}(V_{\widebar m_{k+1}}Y_{\widebar m_{k+1}}V_{\widebar m_{k+1}}^T)\|_F^2 \nonumber\\
 =&\|T_{\widebar m_{k+1}}Y_{\widebar m_{k+1}}+Y_{\widebar m_{k+1}}T_{\widebar m_{k+1}}^T-Y_{\widebar m_{k+1}}B_{\widebar m_{k+1}}B_{\widebar m_{k+1}}^TY_{\widebar m_{k+1}}+E_1\pmb{\gamma\gamma}^TE_1^T\|_F^2\\
 &+2\|Y_{\widebar m_{k+1}}\underline{T}_{\widebar m_{k+1}}^TE_{\widebar m_{k+1}+1}\|_F^2. \notag
\end{align}
Moreover, this value can be used in the next iteration as $\alpha_{k+1}$ if necessary.

The complete implementation\footnote{Many subscripts have been omitted to make the algorithm more readable.} of our new iterative framework is summarized in Algorithm~\ref{PNK_EK_algorithm} where the residual norms of the Riccati operator $\|\mathcal{R}(\cdot)\|_F$ are cheaply computed as in \eqref{Riccati_residualnorm}.
\setcounter{AlgoLine}{0}
\begin{algorithm}
\caption{Projected Newton-Kleinman method with extended Krylov (PNK\_EK).\label{PNK_EK_algorithm}}
\SetKwInOut{Input}{input}\SetKwInOut{Output}{output}
\Input{$A\in\mathbb{R}^{n\times n},$ $A<0$, $B\in\mathbb{R}^{n\times p}$, $C\in\mathbb{R}^{q\times n}$, $m_{\max}$, $\epsilon>0$, $\widebar \eta\in(0,1)$ $\alpha\in(0,1-\widebar \eta).$}
\Output{$P_{k+1}\in\mathbb{R}^{n\times t}$, $t\ll n$, $P_{k+1}P_{k+1}^T=X_{k+1}\approx X$ approximate solution to \eqref{eq.Riccati}.}
\BlankLine
   \nl Set $Y_0=O_{2q}$, $\widebar m=1$ and $k=0$\\
  \nl  Perform economy-size QR, $[C^T,A^{-1}C^T]=[\mathcal{V}_1^{(1)},\mathcal{V}_1^{(2)}][
  {\pmb \gamma}, {\pmb \theta}]$, ${\pmb \gamma},{\pmb \theta}\in\RR^{2q\times q}$\\                                                                                          \nl Set $V_1= [\mathcal{V}_1^{(1)},\mathcal{V}_1^{(2)}]$ \\ 
  \nl Select $\eta_0\in(0,\widebar \eta)$\\
  \For{$m=1, 2,\dots,$ till $m_{\max}$}{
  \nl Compute next basis block $\mathcal{V}_{m+1}$ as in \cite{Simoncini2007} and set $V_{m+1}=[V_{{m}},\mathcal{V}_{m+1}]$ \\
  \nl  Update $T_{m}=V_{m}^TAV_{m}$ as in \cite{Simoncini2007} and $B_{m}=V_{m}^TB$  \\
  \nl Set $Q_m:=T_{m}-\mbox{diag}(Y_k,O_{2q(m-\widebar m)})B_mB_m^T$ \\
  \nl Solve $$Q_m\widetilde Y+
  \widetilde YQ_m^T=-\mbox{diag}(Y_k,O_{2q(m-\widebar m)})B_mB_m^T\mbox{diag}(Y_k,O_{2q(m-\widebar m)})-E_1\pmb{\gamma\gamma}^TE_1^T$$\label{line_alg_innersolve} 
  \If{$\sqrt{2}\|\widetilde Y\underline{T}_{m}^TE_{m+1}\|_F\leq\eta_k\|\mathcal{R}(V_{\widebar m}Y_kV_{\widebar m}^T)\|_F$\label{Riccati_res1}}{
  \nl Compute the coefficients \eqref{polynomial_coefficients} as in Proposition~\ref{polynomial_coefficients_computation}\\
  \nl Compute $\lambda_k$ as in \eqref{stepsizeselection}\\
  \nl Set $Y_{k+1}=(1-\lambda_k)\cdot\mbox{diag}(Y_k,O_{2q(m-\widebar m)})+\lambda_k\widetilde Y$\\ 
  \If{$\|\mathcal{R}(V_{m}Y_{k+1}V_{m}^T)\|_F<\epsilon\cdot\|C^TC\|_F$\label{Riccati_res2}}{ 
\nl \textbf{Break} and go to \textbf{15} }
  \nl Set $k=k+1$ and $\widebar m=m$ \\
  \nl Select $\eta_k\in(0,\widebar \eta)$
  }
  }
  \nl Factorize $Y_{k+1}$ and retain $\widehat Y_{k+1}\in\mathbb{R}^{2mq\times t}$, $t\leq 2mq$\\
  \nl Set $P_{k+1}=V_{m}\widehat Y_{k+1}$\\
\end{algorithm}
%
%
Moreover, as suggested in \cite{Benner2016}, the parameter
$\eta_k$ is given by $\eta_k=1/(k^3+1)$ or $\eta_k=\min\{0.1,0.9\cdot\|\mathcal{R}(V_mY_kV_m^T)\|_F\}$. These values lead to superlinear convergence and quadratic convergence,  respectively.

To reduce the computational efforts of Algorithm~\ref{PNK_EK_algorithm}, one can solve the projected equation in line~\ref{line_alg_innersolve} only periodically, say every $d\geq 1$ iterations. From our numerical experience, we think that this strategy may pay off if implemented only for large $k$, e.g., $k>3$, when $\eta_k\|\mathcal{R}(V_mY_kV_m^T)\|_F$ is small and a quite large space is in general necessary to reach the accuracy prescribed for the current Lyapunov equation. For small $k$, very few iterations are sufficient for the solution of the related equations and performing line~\ref{line_alg_innersolve}, and thus checking the Lyapunov residual norm, only periodically can lead to the execution of unnecessary iterations with a consequent waste of computational efforts in the solution of the linear systems for the basis generation.
However, in all the reported results in section~\ref{Numerical examples} we solve the projected equation at each iteration, i.e., $d=1$.

If the coefficient matrix $A$ is neither negative definite nor stable, we need an initial guess $X_0$ 
such that $A-X_0BB^T$ is stable. Such an $X_0$ exists thanks to Assumption~\ref{main_assumption} and the first
equation to be solved in the Newton sequence \eqref{Newton_step} is 
\begin{equation}\label{firstNewtoneq_X0}
 (A-X_0BB^T)X_1+X_1(A-X_0BB^T)^T=-X_0BB^TX_0-C^TC.
\end{equation}
Once again,
in order to apply a projection method to equation \eqref{firstNewtoneq_X0}, we need to suppose that the matrix
$A-X_0BB^T$ is negative definite, or at least that its projected version is stable.

Supposing that such an $X_0$ is given and low-rank, i.e., 
$X_0=S_0S_0^T$, the same argument of Theorem~\ref{Theorem1}
 shows that the $(k+1)$-th iterate of the Newton-Kleinman method
can be approximated by a matrix $X_{k+1}=S_{k+1}S_{k+1}^T$ such that 
$\mbox{Range}(S_{k+1})\subseteq\mathbf{EK}_{\widebar m_{k+1}}^\square(A,[C^T,S_0])$.
Similarly, for the Newton-Kleinman method equipped with a line search, with the notation of Corollary~\ref{Cor1}, we can show that 
$X_{k+1}=P_{k+1}P_{k+1}^T$ is such that 
$\mbox{Range}(P_{k+1})\subseteq\mathbf{EK}_{\widebar m_{k+1}}^\square(A,[C^T,S_0])$.
Therefore, we can still use our new projection framework and the only modification to Algorithm~\ref{PNK_EK_algorithm} consists in replacing the starting block $C^T$ with $[C^T,S_0]$ and setting $Y_0=V_1^TX_0V_1$.

If the instability of $A$ is due to $\ell$ eigenvalues, $\ell\ll n$, with positive real part, a low-rank stabilizing $X_0$ such that $\text{rank}(X_0)\leq\ell$ can be computed, e.g., by following the procedure presented in
\cite[Section 2.7]{Baensch2015}.

\subsection{The rational Krylov subspace}\label{The Rational Krylov subspace}

The same results stated in Theorem~\ref{Theorem1} and Corollary~\ref{Cor1} can be shown also when the rational Krylov subspace~\eqref{def.rational} is employed as approximation space. The proofs follow the same line of the proofs of Theorem~\ref{Theorem1} and Corollary~\ref{Cor1}. The only technical difference is the presence of the shifts 
$\mathbf{s}=[s_2,\ldots,s_m]^T$. To show that $\mathbf{K}_{m_{k+1}}^\square(A-V_{\widebar m_k}\Theta_kB^T,[C^T,V_{\widebar m_k}\Theta_k],\mathbf{s}_{k+1})\subseteq \mathbf{K}_{\widebar m_{k+1}}^\square(A,C^T, \mathbf{\widebar s}_{k+1})$ for a sufficiently large $\widebar m_{k+1}$, the $m_{k+1}-1$ shifts
collected in $\mathbf{s}_{k+1}$ must constitute a subset of the shifts in $\mathbf{\widebar s}_{k+1}$. More precisely, 
if
$\mathbf{\widebar s}_{k}\in\mathbb{C}^{\widebar m_k-1}$ is such that $\mathbf{K}_{m_{k}}^\square(A-V_{\widebar m_{k-1}}\Theta_{k-1}B^T,[C^T,V_{\widebar m_{k-1}}\Theta_{k-1}],\mathbf{s}_{k})\subseteq \mathbf{K}_{\widebar m_{k}}^\square(A,C^T, \mathbf{\widebar s}_{k})$,
$\mathbf{\widebar s}_{1}=\mathbf{s}_{1}$, we must have
$\mathbf{\widebar s}_{k+1}=[\mathbf{\widebar s}_{k}^T,\mathbf{s}_{k+1}^T,s_{\widebar m_k+m_{k+1}-1},\ldots,s_{\widebar m_{k+1}}]^T\in\mathbb{C}^{\widebar m_{k+1}-1}$, $s_{\widebar m_k+m_{k+1}-1},\ldots,s_{\widebar m_{k+1}}\in\mathbb{C}$.

Provided the matrices $V_{\widebar m_k}$
and $V_{\widebar m_{k+1}}$ are such that $\text{Range}(V_{\widebar m_k})=\mathbf{K}_{\widebar m_{k}}^\square(A,C^T, \mathbf{\widebar s}_{k})$ and $\text{Range}(V_{\widebar m_{k+1}})=\mathbf{K}_{\widebar m_{k+1}}^\square(A,C^T, \mathbf{\widebar s}_{k+1})$, also for the rational Krylov subspace approach the projection of the $(k+1)$-th Lyapunov equation~\eqref{Newton_step} of the Newton-Kleinman scheme can be written as in \eqref{eq.kprojected}.
However, the computation of the residual norm requires a different formulation than the one stated in Proposition~\ref{Prop_projectedeqs} as the Arnoldi relation  \eqref{Arnoldi_rel} no longer holds. Indeed, 
for the rational Krylov subspace $\mathbf{K}_{\widebar m_{k+1}}^\square(A,C^T, \mathbf{\widebar s}_{k+1})$, we have
\begin{equation}
\begin{array}{rll}
AV_{\widebar m_{k+1}}& = & V_{\widebar m_{k+1}}T_{\widebar m_{k+1}}+\mathcal{V}_{\widebar m_{k+1}+1}E^T_{\widebar m_{k+1}+1}\underline{H}_{\widebar m_{k+1}}(\mbox{diag}(s_2,\ldots,s_{\widebar m_{k+1}+1})\otimes I_q)H_{\widebar m_{k+1}}^{-1} \\
&&\\
&&-(I-V_{\widebar m_{k+1}}V_{\widebar m_{k+1}}^T)A\mathcal{V}_{\widebar m_{k+1}+1}E^T_{\widebar m_{k+1}+1}\underline{H}_{\widebar m_{k+1}}H_{\widebar m_{k+1}}^{-1},
\end{array}
 \end{equation}
where the matrix $\underline{H}_{\widebar m_{k+1}}\in\mathbb{R}^{q(\widebar m_{k+1}+1)\times q\widebar m_{k+1}}$
collects the orthonormalization coefficients stemming from the orthogonalization steps and $H_{\widebar m_{k+1}}\in\mathbb{R}^{q\widebar m_{k+1}\times q\widebar m_{k+1}}$
is its principal square submatrix. See, e.g., \cite{Ruhe1994,Druskin2011}. Nevertheless, the residual norm can be still computed at low cost as it is shown in the next proposition.
\begin{Prop}\label{rational_Prop}
 Consider $V_{\widebar m_{k+1}}\in\RR^{n\times qm_{k+1}}$, $\mbox{Range}(V_{\widebar m_{k+1}})=\mathbf{K}_{\widebar m_{k+1}}^\square(A,C^T,\mathbf{s})$, and let $Y_{\widebar m_{k+1}}$ be the solution of the projected equation~\eqref{eq.kprojected}. Then
\begin{multline}
\|(A-X_{k}BB^T)(V_{\widebar m_{k+1}}Y_{\widebar m_{k+1}}V_{\widebar m_{k+1}}^T)+
(V_{\widebar m_{k+1}}Y_{\widebar m_{k+1}}V_{\widebar m_{k+1}}^T)(A-X_{k}BB^T)^T+X_{k}BB^TX_{k}+C^TC\|_F=\\
=\|F_{\widebar m_{k+1}}JF_{\widebar m_{k+1}}^T\|_F,  
\end{multline}
where $F_{\widebar m_{k+1}}$ is the $2q\times 2q$ upper triangular matrix in the ``skinny'' QR factorization of 
$$U_{\widebar m_{k+1}}=[V_{\widebar m_{k+1}}Y_{\widebar m_{k+1}}H_{\widebar m_{k+1}}^{-T}\underline{H}_{\widebar m_{k+1}}^{T}E_{\widebar m_{k+1}+1},\mathcal{V}_{\widebar m_{k+1}+1}s_{\widebar m_{k+1}+1}-(I-V_{\widebar m_{k+1}}V_{\widebar m_{k+1}}^T)A\mathcal{V}_{\widebar m_{k+1}+1}],$$
and
$$
J=\begin{bmatrix}
                     O_q & I_q \\
                     I_q & O_q \\
                     \end{bmatrix}.
                     $$
\end{Prop}
\begin{proof}
 The proof is the same of \cite[Proposition 4.2]{Druskin2011}.
 \begin{flushright}
 $\square$
\end{flushright}
\end{proof}

Also the derivation of the efficient computation of the line search coefficients in Proposition~\ref{polynomial_coefficients_computation} exploits the Arnoldi relation~\eqref{Arnoldi_rel} so that new expressions are needed if the rational Krylov subspace is employed. 
\begin{Prop}\label{polynomial_coefficients_computation_rational}
  Let the matrix $\widetilde F_{\widebar m_{k+1}}$ be the $2q\times 2q$ upper triangular matrix in the ``skinny'' QR factorization of 
$$\widetilde U_{\widebar m_{k+1}}=[V_{\widebar m_{k+1}}\widetilde Y_{\widebar m_{k+1}}H_{\widebar m_{k+1}}^{-T}\underline{H}_{\widebar m_{k+1}}^{T}E_{\widebar m_{k+1}+1},\mathcal{V}_{\widebar m_{k+1}+1}s_{\widebar m_{k+1}+1}-(I-V_{\widebar m_{k+1}}V_{\widebar m_{k+1}}^T)A\mathcal{V}_{\widebar m_{k+1}+1}],$$
and the columns of $G_{\widebar m_k}\in\RR^{n\times 2q}$ be an orthogonal basis for the range of $U_{\widebar m_k}$.
If the rational Krylov subspace is employed in the solution of \eqref{eq.Riccati},
 then the coefficients in \eqref{polynomial_coefficients} are such that
 \begin{align*}
   \alpha_k
    =&\left\|T_{\widebar m_k}Y_{\widebar m_k}+Y_{\widebar m_k}T_{\widebar m_k}^T-Y_{\widebar m_k}B_{\widebar m_k}B_{\widebar m_k}^TY_{\widebar m_k}+E_1\pmb{\gamma\gamma}^TE_1^T\right\|_F^2+\left\|F_{\widebar m_{k}}JF_{\widebar m_{k}}^T\right\|_F^2,\\
    \\
    \beta_k=&2\left\|\widetilde F_{\widebar m_{k+1}}J\widetilde F_{\widebar m_{k+1}}^T\right\|_F^2,\\
    \\
 \gamma_k=& \left\{ \begin{array}{l}
\left\langle F_{\widebar m_k}JF_{\widebar m_k}^T,
 \widetilde F_{\widebar m_{k+1}}J\widetilde F_{\widebar m_{k+1}}^T\right\rangle_F, \quad \mbox{if } \widebar m_{k+1}=\widebar m_{k}, \\
0, \quad \mbox{if } \widebar m_{k+1}>\widebar m_{k},\\ 
 \end{array}\right.\\    
 \\
 \delta_k=&
\left\|\left(\widetilde Y_{\widebar m_{k+1}}-\mbox{diag}(Y_{\widebar m_{k}},O_{2q(\widebar m_{k+1}-\widebar m_{k})})\right)B_{\widebar m_{k+1}}B_{\widebar m_{k+1}}^T
\left(\widetilde Y_{\widebar m_{k+1}}
-\mbox{diag}(Y_{\widebar m_{k}},O_{2q(\widebar m_{k+1}-\widebar m_{k})})\right)\right\|^2_F,\\
 \end{align*}
$$
 \epsilon_k =\left\{
\begin{array}{l}
\begin{array}{l}
 \left\langle T_{\widebar m_{k}}Y_{\widebar m_{k}},
 +Y_{\widebar m_{k}}T_{\widebar m_{k}}^T
  -Y_{\widebar m_{k}}B_{\widebar m_{k}}B_{\widebar m_{k}}^TY_{\widebar m_{k}}
  +E_1\pmb{\gamma\gamma}^TE_1^T,\right.\\
 \left.\left(\widetilde Y_{\widebar m_{k+1}}-Y_{\widebar m_{k}}\right)B_{\widebar m_{k+1}}B_{\widebar m_{k+1}}^T\left(\widetilde Y_{\widebar m_{k+1}}- Y_{\widebar m_{k}}\right)\right\rangle_F\\
\end{array} \quad \mbox{if } \widebar m_{k+1}= \widebar m_{k},
 
 \\
\\

\begin{array}{l}
 \left\langle T_{\widebar m_{k}}Y_{\widebar m_{k}}
  +Y_{\widebar m_{k}}T_{\widebar m_{k}}^T
 -Y_{\widebar m_{k}}B_{\widebar m_{k}}B_{\widebar m_{k}}^TY_{\widebar m_{k}}
   +E_1\pmb{\gamma\gamma}^TE_1^T,\right.\\
   
\left[I_{2q\widebar m_k},O_{2q\widebar m_k\times 2q(\widebar m_{k+1}-\widebar m_k)}\right] \left(\widetilde Y_{\widebar m_{k+1}}-\mbox{diag}(Y_{\widebar m_{k}},O_{2q(\widebar m_{k+1}-\widebar m_{k})})\right)B_{\widebar m_{k+1}}B_{\widebar m_{k+1}}^T\left(\widetilde Y_{\widebar m_{k+1}} \right. \\
\left. \left. -\mbox{diag}(Y_{\widebar m_{k}},O_{2q(\widebar m_{k+1}-\widebar m_{k})})\right)\left[I_{2q\widebar m_k};O_{2q(\widebar m_{k+1}-\widebar m_k)\times 2q\widebar m_k}\right]\right\rangle_F\\
 
 +\left\langle  F_{\widebar m_k}JF_{\widebar m_k}^T,G_{\widebar m_k}^TV_{\widebar m_{k+1}}\left(\widetilde Y_{\widebar m_{k+1}}-\mbox{diag}(Y_{\widebar m_{k}},O_{2q(\widebar m_{k+1}-\widebar m_{k})})\right)B_{\widebar m_{k+1}}B_{\widebar m_{k+1}}^T\left(\widetilde Y_{\widebar m_{k+1}} \right.\right. \\
 \left.\left. - \mbox{diag}(Y_{\widebar m_{k}},O_{2q(\widebar m_{k+1}-\widebar m_{k})})\right)V_{\widebar m_{k+1}}^TG_{\widebar m_{k}}\right\rangle_F,\\
\end{array} \, \mbox{if } \widebar m_{k+1}> \widebar m_{k},
\end{array}\right.
$$
and
$$
 \zeta_k=0.
$$

\end{Prop}

\begin{proof}
 The proof follows the same line of the proof of Proposition~\ref{polynomial_coefficients_computation}.
 In the latter we deeply exploit the orthogonality of 
 $\mathcal{V}_{\widebar m_{k+1}+1}$ with respect to 
 $V_{\widebar m_{k+1}}$. Here we do the same noticing the space spanned by 
 $\mathcal{V}_{\widebar m_{k+1}+1}E^T_{\widebar m_{k+1}+1}\underline{H}_{\widebar m_{k+1}}(\mbox{diag}(s_2,\ldots,s_{\widebar m_{k+1}+1})\otimes I_q)H_{\widebar m_{k+1}}^{-1} 
-(I-V_{\widebar m_{k+1}}V_{\widebar m_{k+1}}^T)A\mathcal{V}_{\widebar m_{k+1}+1}E^T_{\widebar m_{k+1}+1}\underline{H}_{\widebar m_{k+1}}H_{\widebar m_{k+1}}^{-1}$ is orthogonal to $\mbox{Range}(V_{\widebar m_{k+1}}).$ \hfill $\square$


\end{proof}

Even though the coefficients are computed in a different manner, Proposition~\ref{existence_lambda} still holds as $\gamma_k=\zeta_k=0$ for $\widebar m_{k+1}>\widebar m_k$. Therefore, the exact line search \eqref{stepsizeselection} can be carried out also when the rational Krylov subspace is employed.

The projected Newton-Kleinman method with rational Krylov as approximation subspace differs from Algorithm~\ref{PNK_EK_algorithm} only in the basis construction,
the update of the matrix $T_m$ and the computation of the residual norms and the line search coefficients. See Algorithm~\ref{PNK_RK_algorithm}.

\setcounter{AlgoLine}{0}
\begin{algorithm}
\caption{Projected Newton-Kleinman method with rational Krylov (PNK\_RK).\label{PNK_RK_algorithm}}
\SetKwInOut{Input}{input}\SetKwInOut{Output}{output}
\Input{$A\in\mathbb{R}^{n\times n},$ $A<0$, $B\in\mathbb{R}^{n\times p}$, $C\in\mathbb{R}^{q\times n}$, $m_{\max}$, $\epsilon>0$, $\widebar\eta\in(0,1)$, $\alpha\in(0,1-\widebar\eta),$ $\{s_2,\ldots,s_{m_{\max}}\}.$}
\Output{$P_{k+1}\in\mathbb{R}^{n\times t}$, $t\ll n$, $P_{k+1}P_{k+1}^T=X_{k+1}\approx X$ approximate solution to \eqref{eq.Riccati}.}
\BlankLine
   \nl Set $Y_0=O_{q}$, $\widebar m=1$ and $k=0$\\
  \nl  Perform economy-size QR of $C^T$, $C^T=V_1 {\pmb \gamma}$. Set $\mathcal{V}_1\equiv V_1$ \\ 
  \nl Select $\eta_0\in(0,\widebar \eta)$\\
  \For{$m=1, 2,\dots,$ till $m_{\max}$}{
  \nl Compute next basis block $\mathcal{V}_{m+1}$ as in \cite{Druskin2011} and set $V_{m+1}=[V_{{m}},\mathcal{V}_{m+1}]$ \\
  \nl Collect the orthonormalization coefficients in 
  $\underline{H}_m\in\RR^{q(m+1)\times qm}$ \\
  \nl  Update $T_{m}=V_{m}^TAV_{m}$ as in \cite{Druskin2011} and $B_{m}=V_{m}^TB$  \\
  \nl Set $Q_m:=T_{m}-\mbox{diag}(Y_k,O_{q(m-\widebar m)})B_mB_m^T$ \\
  \nl Solve $$Q_m\widetilde Y+
  \widetilde YQ_m^T=-\mbox{diag}(Y_k,O_{q(m-\widebar m)})B_mB_m^T\mbox{diag}(Y_k,O_{q(m-\widebar m)})-E_1\pmb{\gamma\gamma}^TE_1^T$$ 
  \nl  Perform economy-size QR, $\widetilde U=\widetilde G\widetilde F$ of
  $$\widetilde U=[V_m\widetilde YH_m^{-T}\underline{H}_{m}^{T}E_{m+1},\mathcal{V}_{m+1}s_{m+1}-(I-V_{m}V_{m}^T)A\mathcal{V}_{m+1}]$$
  
  \If{$\|\widetilde FJ\widetilde F^T\|_F\leq\eta_k\|\mathcal{R}(V_{\widebar m}Y_kV_{\widebar m}^T)\|_F$}{
  \nl Compute the coefficients \eqref{polynomial_coefficients} as in Proposition~\ref{polynomial_coefficients_computation_rational}\\
  \nl Compute $\lambda_k$ as in \eqref{stepsizeselection}\\
  \nl Set $Y_{k+1}=(1-\lambda_k)\cdot\mbox{diag}(Y_k,O_{q(m-\widebar m)})+\lambda_k\widetilde Y$\\ 
  \If{$\|\mathcal{R}(V_{m}Y_{k+1}V_{m}^T)\|_F<\epsilon\cdot\|C^TC\|_F$}{ 
\nl \textbf{Break} and go to \textbf{16} }
  \nl Set $k=k+1$ and $\widebar m=m$ \\
  \nl Select $\eta_k\in(0,\widebar\eta)$
  }
  }
  \nl Factorize $Y_{k+1}$ and retain $\widehat Y_{k+1}\in\mathbb{R}^{mq\times t}$, $t\leq mq$\\
  \nl Set $P_{k+1}=V_{m}\widehat Y_{k+1}$\\
\end{algorithm}

In practice the shifts defining the rational Krylov subspace can be computed on the fly following the approach presented in \cite[Section 2]{Druskin2011}. This strategy
requires two values $s_0^{(1)}$ and $s_0^{(2)}$, and their complex conjugates, which define an approximation, possibly not very accurate, of 
the spectral region of $-(A-X_0BB^T)$.
At the $m$-th iteration, if $\lambda_1,\ldots,\lambda_{qm}$ denote the eigenvalues of $T_{m}-\mbox{diag}(Y_k,O_{q(m-\widebar m)})B_mB_m^T$,
the $(m+1)$-th shift is computed as
$$s_{m+1}=\argmax_{s\in\mathfrak{S}_m}\frac{1}{|r_m(s)|},$$
where $r_m(s)=\prod_{j=1}^{qm}\prod_{i=2}^m\frac{s-\lambda_j}{s-s_i}$ and $\mathfrak{S}_m$ denotes the convex hull of $\{-\lambda_1,\ldots,-\lambda_{qm}, s_0^{(1)},\widebar s_0^{(1)}, s_0^{(2)}, \widebar s_0^{(2)}\}$. See \cite[Section 2]{Druskin2011} for further details.

The employment of the eigenvalues of $T_{m}-\mbox{diag}(Y_k,O_{q(m-\widebar m)})B_mB_m^T$ in the computation of $s_{m+1}$ is natural in our framework as, at the $m$-th iteration, we are actually trying to solve a Lyapunov equation of the  same form of \eqref{Newton_step}.
We think that this is somehow related to the analysis Simoncini presented in \cite{Simoncini2016a} where a ``pure'' rational Krylov subspace method for the solution of \eqref{eq.Riccati} is studied. In \cite{Simoncini2016a}, if $V_m$ denotes the orthonormal basis of $\mathbf{K}_m^\square(A,C^T,\mathbf{s})$, the eigenvalues of $V_m^TAV_m-Y(V_m^TB)(B^TV_m)$ are employed for computing the $(m+1)$-th shift where $Y$ denotes the solution of the projection of \eqref{eq.Riccati} onto the current subspace, namely $Y$ is such that
$$(V_m^TAV_m)Y+Y(V_m^TAV_m)^T-Y(V_m^TB)(B^TV_m)Y+(V_m^TC^T)(CV_m)=0.
$$
It may be interesting to study the connection between the two approaches and this may help to better understand the convergence properties of projection methods for algebraic Riccati equations. Indeed, to the best of our knowledge, no a-priori result about the stabilizing property of the numerical solution computed by pure projection techniques is available in the literature. See, e.g., \cite[Section 1]{Simoncini2014}. The only result about this topic we are aware of is \cite[Corollary 4.3]{Simoncini2016a}. By exploiting certain perturbation theory results, Simoncini shows that the approximate solution computed by projection methods is stabilizing if it essentially achieves a sufficiently small error\footnote{The result in \cite[Corollary 4.3]{Simoncini2016a} is very general and it is not restricted to projection methods. It can be actually applied to any solution procedure.}. However, neither the error matrix nor the norm of the inverse of the closed-loop Lyapunov operator needed in the estimate presented in \cite[Corollary 4.3]{Simoncini2016a} are computable. See \cite{Simoncini2016a} for further details. Nevertheless, projection methods often produce a solution $X$ such that $A-XBB^T$ is stable in practice. See, e.g., \cite[Section 4]{Jbilou2003}. 


\section{Enhanced convergence results}\label{Convergence results}
The convergence of the inexact Newton-Kleinman method has been proved under some suitable assumptions in, e.g., \cite{Benner2016}. The main drawback of \cite[Theorem 10]{Benner2016} is that the authors have to assume the matrix $A-X_{k+1}BB^T$ to be stable for $k\geq k_0$. We show that, in our framework, $A-X_{k+1}BB^T$ is stable 
by construction and the sequence
  $\{X_k\}_{k\geq 0}$ computed by the projected Newton-Kleinman method is well-defined if $A<0$, $\|L_{k+1}\|_F$ is sufficiently small, and 
$BB^T$ and $C^TC$ fulfill certain conditions.
 In the following theorem, given a matrix $D\in\mathbb{R}^{n\times n}$, $\text{Ker}(D)$ denotes the kernel of $D$, namely $\text{Ker}(D)=\{y\in\mathbb{R}^n\text{ s.t. }Dy=0\}$.
%
%
%
%

 \begin{theorem}\label{Theorem_conv2}
 Assume $A<0$ and suppose that  either
 \begin{itemize}
  \item[(i)] $\text{Ker}(C^TC)\subseteq \text{Ker}(BB^T)$ and  $\|L_{k+1}\|_F\leq \sigma_q^2(C),  
  $
 \end{itemize}
or
 \begin{itemize}
  \item[(ii)] $BB^T\leq C^TC$, i.e., $BB^T-C^TC$ is negative semidefinite,
  and  
  $
 \|L_{k+1}\|_F\leq \sigma_p^2(B),  
  $
 \end{itemize}
where $\sigma_q(C)$, $\sigma_p(B)$ denote the smallest 
singular value of $C$ and $B$ respectively. Then, for all $\lambda_k\in(0,2]$, $A-X_{k+1}BB^T$ is stable and the sequence
  $\{X_k\}_{k\geq 0}$ computed by the projected Newton-Kleinman method is well-defined.

 \end{theorem}

 \begin{proof}
 We first assume that $A-X_kBB^T$ and $V_{\widebar m_{k}+ \ell}^T(A-X_kBB^T)V_{\widebar m_{k}+ \ell}$ are stable for any $\ell\geq 0$ and we show that this, along with the hypotheses above, implies the stability of both $A-X_{k+1}BB^T$ and $V_{\widebar m_{k+1}+ \ell}^T(A-X_{k+1}BB^T)V_{\widebar m_{k+1}+ \ell}$ for any $\ell\geq0$.
 
 The matrix $\widetilde X_k=X_k+Z_k=V_{\widebar m_{k+1}}\widetilde Y_{\widebar m_{k+1}}V_{\widebar m_{k+1}}^T$
 in \eqref{Newton_step_inexact} is symmetric positive semidefinite. Indeed, $\widetilde Y_{\widebar m_{k+1}}$ is symmetric positive semidefinite as it is the solution of the projected equation \eqref{eq.kprojected} where the coefficient matrix $Q_{\widebar m_{k+1}}=V_{\widebar m_{k+1}}^T(A-X_kBB^T)V_{\widebar m_{k+1}}$ is stable by assumption as $V_{\widebar m_{k+1}}^T(A-X_kBB^T)V_{\widebar m_{k+1}}=V_{\widebar m_{k}+\widebar \ell}^T(A-X_kBB^T)V_{\widebar m_{k}+\widebar \ell}$ for a certain $\widebar \ell\geq 0$, and the right-hand side is negative semidefinite.
 
 Rearranging equation \eqref{Newton_step_inexact}, we can write
 \begin{align}\label{eq_Theorem_conv}
\left(A-X_{k+1}BB^T\right)(X_k+Z_k)+(X_k+Z_k)\left(A-X_{k+1}BB^T\right)^T=&-C^TC-X_{k+1}BB^TX_{k+1}\notag\\
&+L_{k+1}-\frac{2-\lambda_k}{\lambda_k}(X_{k+1}-X_{k})BB^T(X_{k+1}-X_{k}),  
 \end{align}
 where we have exploited the equality $\lambda_k Z_k=X_{k+1}-X_k$.
 
 We now suppose that there exists an eigenpair $(\vartheta,z)$ of $A-X_{k+1}BB^T$ such that $\text{Re}(\vartheta)\geq 0$ and $\|z\|_F=1$. We suppose $z$ to be a left eigenvector of $A-X_{k+1}BB^T$, namely
 $$z^*(A-X_{k+1}BB^T)=\vartheta z^*.$$
  We first notice that $z\notin\text{Ker}(BB^T)$ otherwise 
 $\vartheta$ would be a point in $W(A)$ as we would have
 $$\vartheta=z^*(A-X_{k+1}BB^T)z=z^*Az,$$
  contradicting the negative definiteness of $A$.
 
 Pre and postmultiply \eqref{eq_Theorem_conv} by $z^*$ and $z$ respectively, we get
 \begin{equation}\label{eq_convergence2}
 2\cdot\text{Re}(\vartheta)z^*(X_k+Z_k)z=z^*Dz,  
 \end{equation}
 where $D:=-C^TC-X_{k+1}BB^TX_{k+1}+L_{k+1}-\frac{2-\lambda_k}{\lambda_k}(X_{k+1}-X_{k})BB^T(X_{k+1}-X_{k})$.
 The left-hand side in the above equation is nonnegative since $\text{Re}(\vartheta)\geq0$ and $X_k+Z_k$ is positive semidefinite so that $z^*Dz\geq0$. 
 Since
 $$
 z^*Dz=-z^*C^TCz+z^*L_{k+1}z-z^*X_{k+1}BB^TX_{k+1}z-\frac{2-\lambda_k}{\lambda_k}z^*(X_{k+1}-X_{k})BB^T(X_{k+1}-X_{k})z,
 $$
  the only scalar that can assume a positive value in the above expression is $z^*L_{k+1}z$ as $C^TC$, $X_{k+1}BB^TX_{k+1}$, and 
  $(X_{k+1}-X_k)BB^T(X_{k+1}-X_k)$ are all positive semidefinite and $\lambda_k\in(0,2]$. 
  
   If the set of assumptions \emph{(i)} holds, $z^*C^TCz\neq 0$ and 
  thanks to the symmetry of $L_{k+1}$ we can write
 $$z^*L_{k+1}z\leq\|L_{k+1}\|_2\leq\|L_{k+1}\|_F\leq\sigma_q^2(C)\leq 
 z^*C^TCz,$$
 where $\|\cdot\|_2$ denotes the spectral norm.

As a result, we have
 \begin{align*}
z^*Dz&=z^*(-C^TC+L_{k+1})z-z^*X_{k+1}BB^TX_{k+1}z-\frac{2-\lambda_k}{\lambda_k}z^*(X_{k+1}-X_{k})BB^T(X_{k+1}-X_{k})z\\  
&=\pi_1+\rho +\frac{2-\lambda_k}{\lambda_k}\varphi,
 \end{align*}
 and the scalars $\pi_1:=z^*(-C^TC+L_{k+1})z$, $\rho:=-z^*X_{k+1}BB^TX_{k+1}z$, $\varphi:=-z^*(X_{k+1}-X_{k})BB^T(X_{k+1}-X_{k})z$ are all nonpositive meaning they are all necessarily zero as the left-hand side in~\eqref{eq_convergence2} is nonnegative.
 
 Similarly, if the assumptions in \emph{(ii)} hold, we can write
 $$z^*L_{k+1}z\leq\|L_{k+1}\|_2\leq\|L_{k+1}\|_F\leq\sigma_p^2(B)\leq 
 z^*BB^Tz,$$
 and
 \begin{align*}
z^*Dz=&z^*(-C^TC+BB^T)z+z^*(L_{k+1}-BB^T)z-z^*X_{k+1}BB^TX_{k+1}z\\
&-\frac{2-\lambda_k}{\lambda_k}z^*(X_{k+1}-X_{k})BB^T(X_{k+1}-X_{k})z 
=\pi_2+\omega_2+\rho +\frac{2-\lambda_k}{\lambda_k}\varphi,
 \end{align*}
 and again all the scalars in the above expression are nonpositive thanks to the assumption $BB^T\leq C^TC$.
 
 In particular, in both cases we have $\varphi=-z^*(X_{k+1}-X_{k})BB^T(X_{k+1}-X_{k})z=0$, which means
 $z^*(X_{k+1}-X_{k})B= 0$ so that 
 $$z^*X_{k+1}B=z^*X_{k}B.$$

 To conclude,
 $$\vartheta z^* =z^*\left(A-X_{k+1}BB^T\right)=z^*\left(A-X_{k}BB^T\right),
 $$
 and $\vartheta$ is thus an eigenvalue of $A-X_{k}BB^T$ which contradicts the stability of $A-X_{k}BB^T$.
 
 It is easy to show that the stability of $A-X_{k+1}BB^T$ is necessarily maintained when this matrix is projected onto the current subspace, namely $V_{\widebar m_{k+1}}^T(A-X_{k+1}BB^T)V_{\widebar m_{k+1}}$ is stable since $V_{\widebar m_{k+1}}^TL_{k+1}V_{\widebar m_{k+1}}=O$ . However, to ensure the sequence $\{X_k\}_{k\geq 0}$ to be well-defined, this property has to be maintained also when the space is expanded, i.e., the matrix 
$V_{\widebar m_{k+1}+\ell}^T(A-X_{k+1}BB^T)V_{\widebar m_{k+1}+\ell}$ has to be stable for any $\ell\geq0$. 

If $Q:=V_{\widebar m_{k+1}+\ell}^T(A-X_{k+1}BB^T)V_{\widebar m_{k+1}+\ell}$,
 projecting~\eqref{eq_Theorem_conv} onto $\text{Range}(V_{\widebar m_{k+1}+\ell})$ yields
\begin{align*}
Q \text{diag}(\widetilde Y_{\widebar m_{k+1}},O_r)+\text{diag}(\widetilde Y_{\widebar m_{k+1}},O_r)Q^T=&-V_{\widebar m_{k+1}+\ell}^TC^TCV_{\widebar m_{k+1}+\ell}+V_{\widebar m_{k+1}+\ell}^TL_{k+1}V_{\widebar m_{k+1}+\ell}\\
&-\text{diag}(Y_{\widebar m_{k+1}},O_r)B_{\widebar m_{k+1}+\ell}B_{\widebar m_{k+1}+\ell}^T\text{diag}(Y_{\widebar m_{k+1}},O_r)\\
&-\frac{2-\lambda_k}{\lambda_k}(\text{diag}(Y_{\widebar m_{k+1}},O_r)-\text{diag}(Y_{\widebar m_{k}},O_{r+s}))B_{\widebar m_{k+1}+\ell}\cdot \\
&\cdot B_{\widebar m_{k+1}+\ell}^T(\text{diag}(Y_{\widebar m_{k+1}},O_r)-\text{diag}(Y_{\widebar m_{k}},O_{r+s})),
\end{align*}
where $r=\text{dim}(\text{Range}(V_{\widebar m_{k+1}+\ell}))-\text{dim}(\text{Range}(V_{\widebar m_{k+1}}))$ and 
$s=\text{dim}(\text{Range}(V_{\widebar m_{k+1}}))-\text{dim}(\text{Range}(V_{\widebar m_{k}}))$\footnote{The value of $r$ and $s$ depends on the adopted approximation space. For the extended Krylov subspace approach we have $r=2\ell q$ and $s=2q$ whereas $r=\ell q$ and $s=q$ if the rational Krylov subspace method is employed.}.
The same arguments as before demonstrate the stability of
$V_{\widebar m_{k+1}+\ell}^T(A-X_{k+1}BB^T)V_{\widebar m_{k+1}+\ell}$. Indeed, $\text{diag}(\widetilde Y_{\widebar m_{k+1}},O_r)$ is still symmetric positive semidefinite and for any eigenvector $w$ of $V_{\widebar m_{k+1}+\ell}^T(A-X_{k+1}BB^T)V_{\widebar m_{k+1}+\ell}$ associated to an eigenvalue $\omega$ with nonnegative real part, it must hold $BB^T
V_{\widebar m_{k+1}+\ell}w\neq 0$, otherwise
$$\omega=w^*V_{\widebar m_{k+1}+\ell}^T(A-X_{k+1}BB^T)V_{\widebar m_{k+1}+\ell}w=w^*V_{\widebar m_{k+1}+\ell}^TAV_{\widebar m_{k+1}+\ell}w,$$
which again contradicts the negative definiteness of $A$.

Moreover, if \emph{(i)} holds, then 
$$\text{Ker}(V_{\widebar m_{k+1}+\ell}^TC^TCV_{\widebar m_{k+1}+\ell})\subseteq\text{Ker}(V_{\widebar m_{k+1}+\ell}^TBB^TV_{\widebar m_{k+1}+\ell}),$$
and for any vector $z\in\mathbb{R}^{\text{dim}(\text{Range}(V_{\widebar m_{k+1}+\ell}))}$ such that $\|V_{\widebar m_{k+1}+\ell}z\|_F=1$, we have
 $$z^*V_{\widebar m_{k+1}+\ell}^TL_{k+1}V_{\widebar m_{k+1}+\ell}z\leq\|L_{k+1}\|_F\leq\sigma_q^2(C)\leq 
 z^*V_{\widebar m_{k+1}+\ell}^TC^TCV_{\widebar m_{k+1}+\ell}z.$$
 Similarly, if \emph{(ii)} holds, then $V_{\widebar m_{k+1}+\ell}^TBB^TV_{\widebar m_{k+1}+\ell}\leq V_{\widebar m_{k+1}+\ell}^TC^TCV_{\widebar m_{k+1}+\ell}$ and 
 $$z^*V_{\widebar m_{k+1}+\ell}^TL_{k+1}V_{\widebar m_{k+1}+\ell}z\leq\|L_{k+1}\|_F\leq\sigma_p^2(B)\leq 
 z^*V_{\widebar m_{k+1}+\ell}^TBB^TV_{\widebar m_{k+1}+\ell}z.$$

 \begin{flushright}
 $\square$
\end{flushright}
 \end{proof}

 Notice that the assumptions \emph{(i)} and \emph{(ii)} in Theorem~\ref{Theorem_conv2} are not difficult to meet in practice. For instance, in case of Riccati equations coming from balancing based model order reduction techniques for linear systems, the matrices $BB^T$, $C^TC$ are often related to the (discrete) indicating functions of two portions of the domain of interest $\Omega$: $\Omega_B$, where the controller is active, and $\Omega_C$, where we want to observe the output. Therefore, if $\Omega_B\subseteq\Omega_C$, then  $\text{Ker}(C^TC)\subseteq\text{Ker}(BB^T)$. Moreover, the condition $BB^T\leq C^TC$  can be sometimes obtained by rescaling the input and/or the output function(s) of the underlined system.
 In both cases \emph{(i)} and\emph{(ii)}, Algorithm~\ref{PNK_EK_algorithm} and \ref{PNK_RK_algorithm} can be easily modified to attain a sufficiently small $\|L_{k+1}\|_F$. This can be done by expanding the approximation space until $\|L_{k+1}\|_F\leq\min\{ \sigma_q^2(C),\eta_k\|\mathcal{R}(X_k)\|_F\}$ or $\|L_{k+1}\|_F\leq\min\{\sigma_p^2(B),\eta_k\|\mathcal{R}(X_k)\|_F\}$ depending on which condition in Theorem~\ref{Theorem_conv2} has to be fulfilled.
 
 We would like to underline the fact that, to the best of our knowledge, our routines are the only low-rank methods for the solution of large-scale Riccati equations that are guaranteed to compute a stabilizing numerical solution if the coefficient matrices $A$, $B$ and $C$ fulfill the assumptions in  
 Theorem~\ref{Theorem_conv2}.
 The other solution schemes available in the literature ensure the computation of a stabilizing solution only \emph{asymptotically}. For instance, pure projection techniques provide a stabilizing solution if the employed approximation spaces span the whole $\mathbb{R}^n$ since the computed solution coincide with the exact stabilizing solution $X$ in this case. However, this scenario is not numerically feasible. See also the discussion at the end of section~\ref{The Rational Krylov subspace}. Similarly, it can be shown that RADI \cite{Benner2018} converges to the exact stabilizing solution $X$ by exploiting its equivalence with ILRSI \cite{Lin2015}. However, this does not guarantee that the actual computed low-rank approximate solution is stabilizing in practice.

We conclude this section by showing that the exact line search we perform leads to a monotonic decrease in $\|\mathcal{R}(X_k)\|_F$.
 
 \begin{Cor}
  With the assumptions of Theorem~\ref{Theorem_conv2}, if the sequence of step sizes $\lambda_k$ is uniformly bounded from below, then the residual norms $\|\mathcal{R}(X_k)\|_F$ decrease monotonically to zero.
 \end{Cor}

 \begin{proof}
  The proof directly comes from the proof of \cite[Theorem 7]{Benner1998}. In \cite[Theorem 7]{Benner1998} the authors assume $(A,B)$ to be a controllable pair to ensure the sequence of iterates $X_k$ to be bounded.
  In \cite[Lemma 2.3]{Guo1999} Guo and Laub show that $\{X_k\}_{k\geq0}$  is bounded 
  if $(A,B)$ is only stabilizable provided  $A-X_kBB^T$ is stable for all $k$. In Assumption~\ref{main_assumption} we have assumed $(A,B)$ to be stabilizable 
 while the stability of $A-X_kBB^T$ has been shown in Theorem~\ref{Theorem_conv2}.
 \begin{flushright}
 $\square$
\end{flushright}
  \end{proof}

\section{Numerical examples}\label{Numerical examples}
In this section we compare Algorithm~\ref{PNK_EK_algorithm} and \ref{PNK_RK_algorithm} with
state-of-the-art methods for the solution of large-scale algebraic Riccati equations. In particular, our new procedures are compared with the inexact Newton-Kleinman method with ADI as inner solver (iNK+ADI) \cite{Benner2016}, the Newton-Kleinman method with Galerkin acceleration (NK+GP) \cite{Benner2010}, projection methods with extended (EKSM) and rational (RKSM) Krylov subspace and RADI \cite{Benner2018}. The two variants of the Newton-Kleinman method are available in the M-M.E.S.S. package \cite{Saak2016}. A Matlab implementation of projection methods for Riccati equations can be found on the web page of Simoncini\footnote{{\tt http://www.dm.unibo.it/\textasciitilde simoncin/software.html}} while we thank Jens Saak for providing us with the RADI code\footnote{A Matlab implementation will be available in the next version of the M-M.E.S.S. package.}.

The performances of the algorithms are compared in terms of memory requirements and computational time. For the former we report the maximum number of vectors of length $n$ that need to be stored. For instance, in our framework the storage demand consists in the  dimension of the computed subspace. The same for the ``pure'' projection procedures. For the other methods, the memory requirements amount to the number of columns of the low-rank factors of all the current iterates
and we thus check this value at each step, before any low-rank truncation is performed. See \cite{Benner2016, Benner2010, Benner2018} and\cite[Table 3]{Benner2018a} for further details. We also report the rank of the computed solution, the relative residual norm achieved with such a solution and the number of (outer) iterations
performed to converge.
In iNK+ADI and in NK+GP, ADI is employed as inner solver for the Lyapunov equations stemming from the Newton scheme. We thus report also the maximum number of ADI iterations as this value is strictly related to the memory consumption of iNK+ADI and NK+GP. See \cite[Table 3]{Benner2018a}. See also, e.g., \cite{Li2004} for further details about ADI.

The RADI method, as its linear counterpart ADI, requires the computation of effective shifts. In \cite{Benner2018} several kinds of shifts $s_j$ are proposed and the performance achieved with different $s_j$'s seems to be highly problem dependent.
However, the residual Hamiltonian shifts, denoted in \cite[Section 5]{Benner2018} by ``Ham, $\ell=2p$''\footnote{In our notation it would be ``Ham, $\ell=2q$'' as $p=\text{rank}(C^T)$ in \cite{Benner2018}.}, provide very good performance 
in all the experiments reported in 
\cite{Benner2018}. We thus employ the same RADI shifts.
When ADI is used as inner solver in iNK+ADI and in NK+GP, the default setting of the M-M.E.S.S. package is used for computing the ADI shifts\footnote{{\tt opts.shifts.method='projection'}, {\tt opts.shifts.l0=max(6,size(B,2))}.}.

In \cite{Benner2016}, two values for the forcing parameter $\eta_k$ are proposed: $\eta_k=1/(1+k^3)$ and $\eta_k=\max\{0.1,0.9\cdot\|\mathcal{R}(V_mY_kV_m)\|_F\}$. These values lead to a superlinear and a quadratic convergence of the Newton scheme respectively. However, in all our numerical experiments, we notice a remarkable increment in both the computational time and the memory requirements of our new procedures when $\eta_k=\max\{0.1,0.9\cdot\|\mathcal{R}(V_mY_kV_m)\|_F\}$ is employed. Indeed, the quadratic convergence obtained is in terms of the number $k$ of Lyapunov equations we need to solve. Even though we have to solve fewer equations, each of them requires to be more accurately solved and, in general, this means that a larger subspace has to be generated and more computational efforts are thus demanded. Therefore, in all the reported experiments, $\eta_k=1/(1+k^3)$.

The tolerance for the final relative residual norm is always set to $10^{-8}$.

All results were obtained with Matlab R2017b \cite{MATLAB} on a Dell machine with 2GHz
processors and 128 GB of RAM.

\begin{num_example}\label{Ex.1}
{\rm
 In the first example we consider a matrix $A$ in \eqref{eq.Riccati} stemming from the centered finite difference discretization of the 3D Laplacian $\mathcal{L}(u)=\Delta u$ on the unit cube with zero Dirichlet boundary condition. In particular, if $T=1/(n_0-1)^2\cdot\mbox{tridiag}(1,-2,1)$ denotes the matrix representing the discrete operator associated to the 1D Laplacian, then 
 $$A=T\otimes I_{n_0}\otimes I_{n_0}+ I_{n_0}\otimes T\otimes I_{n_0}+ I_{n_0}\otimes I_{n_0}\otimes T.$$
 The matrix $A$ is thus symmetric negative definite and we can set $X_0=O$.
 Since all the methods we compare require solving many linear systems with $A$ - or a shifted $A$ - we reorder the entries of this matrix by the Matlab function {\tt amd}.
 
The low-rank matrices $B\in\RR^{n\times p}$ and $C\in\RR^{q\times n}$ have random entries that have been scaled by the mesh size $1/(n_0-1)^2$ to match the magnitude of the components of $A$. In particular,
$B=1/(n_0-1)^2\cdot\mathtt{rand}(n,p)$ where $n=n_0^3$. Similarly for $C$. 

In Table \ref{tab1} we report the results for $n=125000$ and different values of $p$ and $q$.
 
 \begin{table}[!ht]
 \centering
 \caption{Example \ref{Ex.1}. Results for different values of $p$ and $q$. \label{tab1}}
\begin{tabular}{|r| r r r r r|}
 \hline
 & It. (max inner It.)  & Mem. & rank($X$) & Rel. Res & Time (secs) \\
 \hline
 &\multicolumn{5}{c|}{{$p=q=1$}} \\
\hline
 PNK\_EK &  15 (-) & 32 &18  &4.18e-10 &10.19  \\

 PNK\_RK &  17 (-) & 19&17& 7.35e-12&28.10 \\

 EKSM  &  13 (-)  &28 & 24&4.16e-9&10.23\\

 RKSM  &  12 (-)  &14 & 12 &4.58e-9&20.26\\
 
  iNK+ADI  & 2 (8)   &61 & 13 &6.29e-9&27.85\\

    NK+GP  & 2 (32)   &106 & 25 &3.07e-15&120.72\\

 RADI  &  12 (-) &22 & 12 &4.86e-9&20.55\\

\hline 

&\multicolumn{5}{c|}{{$p=q=10$}} \\
\hline
 PNK\_EK &  14 (-)  & 300 &186 &1.36e-9 &37.62  \\

 PNK\_RK & 15 (-)  & 170&150& 1.13e-9&33.65 \\

 EKSM  & 13 (-)   &280 & 234 &5.15e-9&35.85\\

 RKSM  & 14 (-)   &160 & 140 &6.93e-9&31.13\\
 
  iNK+ADI  & 3 (17)    &1020 & 177 &6.23e-10&99.51\\

    NK+GP  & 2 (35)   &1360 &237 &3.60e-11&226.52\\

    RADI  &  14 (-)  &780 & 140 &2.98e-9&33.79\\
    
\hline

&\multicolumn{5}{c|}{{$p=10$, $q=1$}} \\
\hline
 PNK\_EK &  14 (-)   & 30 &16  &1.36e-10 &9.96  \\

 PNK\_RK & 12 (-)  & 14&12& 3.02e-9&21.46 \\

 EKSM  &  11 (-) &24 & 19&3.93e-9&9.53\\

 RKSM  &  11 (-) &13 &11 &4.62e-9&19.74\\
 
  iNK+ADI  & 3 (14)   &537 & 18 &2.45e-10&53.71\\

    NK+GP  & 2 (30)   &702 & 21 &2.06e-11&139.58\\

   RADI  &  12 (-) &40 & 12 &1.60e-9&22.91\\
 
\hline 

\end{tabular}
 \end{table}

 In this example, the iNK+ADI and the NK+GP are not very competitive in terms of both memory requirements and computational time when compared to the other methods. 
 
 The procedures based on the extended Krylov subspace, i.e., PNK\_EK and EKSM, are very fast. Indeed, they need few iterations to converge and the precomputation of the LU factors\footnote{The time for the computation of the LU factorization is included in all the reported results, also for the next examples.} of $A$ makes the linear solves very cheap. Very few iterations are needed also in PNK\_RK and RKSM but these are computationally more expensive due to the presence of different shifts in the linear systems. The gains coming from the precomputation of the LU factors are less outstanding in the case $p=q=10$ to the point that PNK\_RK and RKSM turn out to be faster than the corresponding procedures based on the extended Krylov subspace. This is mainly due to the cost of the inner solves. Indeed, the solution of the projected equation\footnote{This amounts to a Lyapunov equation in case of PNK\_EK and PNK\_RK, and a Riccati equation in case of EKSM and RKSM.} grows cubically with the space dimension and in PNK\_EK and EKSM a quite large space is constructed when $p=q=10$.
 
 The methods based on the rational Krylov subspace demand little storage and provide a very low-rank solution. This is typical also when projection methods are applied to Lyapunov equations. 
 
For all the tested values of $p$ and $q$, both PNK\_EK and PNK\_RK implicitly solve six Lyapunov equations of the Newton scheme \eqref{Newton_step}. 

 The RADI method is very competitive in terms of computational time and its performance is very similar to the ones achieved by PNK\_RK and RKSM. However, it is more memory consuming in general.
 
 For $p=q=1$, we also compare PNK\_EK with the inexact Newton-Kleinman method where each Lyapunov equations of the scheme is solved by K-PIK \cite{Simoncini2007}.
 Such a procedure is called iNK+K-PIK in the following and it solves the $(k+1)$-th Lyapunov equation \eqref{Newton_step_inexact} by projection onto the extended Krylov subspace
 $\mathbf{EK}_m^\square(A-X_kBB^T,[C^T,X_kB])$. In both PNK\_EK and iNK+K-PIK we need to solve six Lyapunov equations to achieve $\|\mathcal{R}(X_{k+1})\|_F/\|C^TC\|_F\leq10^{-8}$ and the relative residual norms produced by the two methods have a very similar trend. See Figure~\ref{Ex.1_Fig.1}. 
 
 \begin{figure}
  \centering    
  \caption{Example~\ref{Ex.1}. Relative residual norms produced by iNK+K-PIK and PNK\_EK for $p=q=1$.} 
  \label{Ex.1_Fig.1}
  	\begin{tikzpicture}
    \begin{semilogyaxis}[width=0.8\linewidth, height=.27\textheight,
      legend pos = north east,
      xlabel = $k$, ylabel = Riccati relative residual norm,xmin=0, xmax=6, ymax = 1e1]
      \addplot+[thick] table[x index=0, y index=1]  {result_Laplacian.dat};
       \addplot+ [thick]table[x index=0, y index=2]  {result_Laplacian.dat};
        \legend{iNK+K-PIK,PNK\_EK};
    \end{semilogyaxis}          
  \end{tikzpicture}
  
\end{figure}
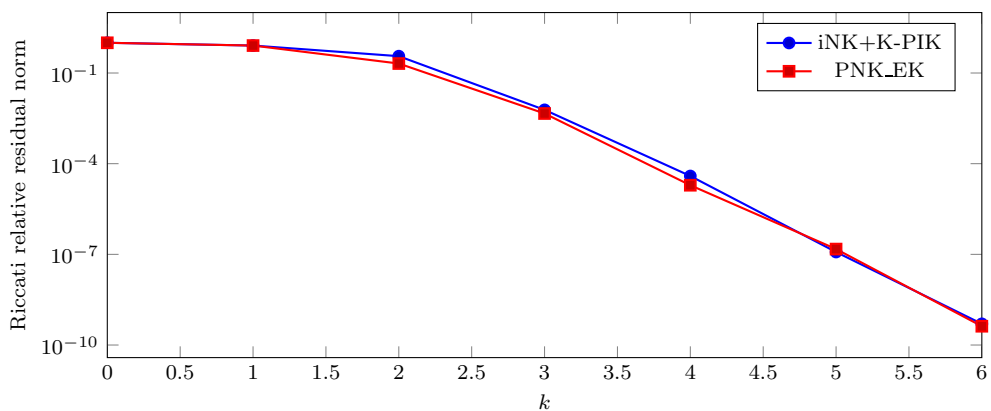

 We want to compare the dimension of the subspaces constructed by iNK+K-PIK to solve the $k+1$ equations of the Newton scheme \eqref{Newton_step_inexact} with the corresponding dimension of $\mathbf{EK}_m^\square(A,C^T)$, i.e., with $2qm=\text{dim}(\mathbf{EK}_m^\square(A,C^T))$ such that $\sqrt{2}\|\widetilde Y \underline{T}_m^TE_{m+1}\|_F\leq \eta_k\|\mathcal{R}(V_{\widebar m}Y_kV_{\widebar m})\|_F$ in Algorithm~\ref{PNK_EK_algorithm} for $k=0,\ldots,5$.
 The results are reported in Table~\ref{tab_2}. 

  \begin{table}[!ht]
 \centering
 \caption{Example \ref{Ex.1}, $p=q=1$. Comparison between the memory consumption of iNK+K-PIK and PNK\_EK. \label{tab_2}}
\begin{tabular}{l| r r r r r r }
 
 & \multicolumn{6}{c}{{$k$}}\\

 & 0 & 1 & 2& 3 & 4& 5 \\

\hline
& & & & & & \\[\dimexpr-\normalbaselineskip+2pt]
 
 $\text{dim}(\mathbf{EK}_m^\square(A-X_kBB^T, [C^T,X_kB]))$ (It.)  &  4 (2) & 8 (2) & 8 (2) & 12 (3) & 24 (6) &36 (9)\\

 $\text{dim}(\mathbf{EK}_m^\square(A,C^T))$ (It.) & 4 (2) & 6 (3) & 8 (4) & 12 (6) & 20 (10)  & 30 (15) \\
\end{tabular}
 \end{table}
 
 If $k>0$, the dimension of $\mathbf{EK}_m^\square(A-X_kBB^T,[C^T,X_kB])$ grows faster than $\text{dim}(\mathbf{EK}_m^\square(A,C^T))$ as four new basis vectors are added to the current space at each iteration instead of only two. This may lead to some redundancy in $\mathbf{EK}_m^\square(A-X_kBB^T,[C^T,X_kB])$ and, at least for this example, a smaller subspace can be constructed to achieve the same level of accuracy in the solution of the $(k+1)$-th equation. For instance, in the solution of the second Lyapunov equation ($k=1$) only one iteration of K-PIK is not sufficient to achieve the prescribed level of accuracy and a second iteration is performed 
 so that the algorithm necessarily ends up constructing a subspace of dimension 8. On the other hand, since only two basis vectors are added to $\mathbf{EK}_m^\square(A,C^T)$ at each iteration, PNK\_EK manages to realize that a space of dimension 6 contains already enough spectral information to solve the second equation. 
 Moreover, the final dimension of $\mathbf{EK}_m^\square(A,C^T)$ is much smaller than the one predicted by Corollary~\ref{Cor1}. Indeed, for this example, the latter amounts to 52 but a subspace of dimension 15 is sufficient to solve the Riccati equation.
 
 Notice that iNK+K-PIK and PNK\_EK are not comparable from a computational time perspective. Indeed, iNK+K-PIK constructs $\mathbf{EK}_m^\square(A-X_kBB^T,[C^T,X_kB])$ from scratch for all $k=0,\ldots,5$ and the computation of the last space $\mathbf{EK}_m^\square(A-X_5BB^T,[C^T,X_5B])$ is more expensive than the overall PNK\_EK procedure.

%

 We conclude this example by showing that the matrix $A-X_kBB^T$ and its projection computed by PNK\_EK are stable for all $k$ as predicted by Theorem~\ref{Theorem_conv2}. To this end we consider a smaller problem, $n=1000$, to be able to allocate the full iterates $X_k$'s. $B$ is as before with $p=1$ and $C^T=B$ so that $\text{Ker}(C^TC)=\text{Ker}(BB^T)$. For this problem setting we need eleven iterations to achieve the prescribed accuracy and six Lyapunov equations are implicitly solved. In Figure~\ref{Ex.1_Fig.2} (left) we plot the real part of the eigenvalues of $A-X_kBB^T$ for all $k=0,\ldots,6$ whereas on the right we report the scaled real part of the rightmost eigenvalue of $V_{m}^T(A-X_kBB^T)V_{m}$ for $k=0,\ldots,5$ and $m=1,\ldots,11$. As predicted by Theorem~\ref{Theorem_conv2}, all the matrices $A-X_kBB^T$ and their projected counterparts are stable. Very similar results are obtained in the case of PNK\_RK and we decide not to report them here.
 
 \begin{figure}   \caption{Example~\ref{Ex.1}. Left: Real part of the eigenvalues of $A-X_kBB^T$, $X_k$ computed by PNK\_EK ($X_0=O$), for all $k=0,\ldots,6$, $n=1000$ and $p=q=1$. The square indicates zero. Right: Scaled real part of the rightmost eigenvalue of $V_{m}^T(A-X_kBB^T)V_{m}$ for $k=0,\ldots,5$ and $m=1,\ldots,11$.} 
  \label{Ex.1_Fig.2}
	\begin{center}
	\begin{tabular}{cc}
		{	\begin{tikzpicture}
    \begin{axis}[width=0.7\linewidth, height=.27\textheight,
      legend pos = north east,
      xlabel = $\text{Re}(\lambda_j(A-X_kBB^T))$, ylabel = $k$,xmin=-0.15,xmax=0.005]
      \addplot+[mark=o, color=blue, mark options={solid}, only marks] table[x index=3, y index=0]  {result_Laplacian_eigenvalues_1.dat};
      \addplot+[mark=square, color=red] table[x index=0, y index=0]  {result_Laplacian_eigenvalues_1.dat};
      \addplot+[mark=o, color=blue, mark options={solid}, only marks] table[x index=4, y index=1]  {result_Laplacian_eigenvalues_1.dat};
      \addplot+[mark=square, color=red] table[x index=0, y index=1]  {result_Laplacian_eigenvalues_1.dat};
      \addplot+[mark=o, color=blue, mark options={solid}, only marks] table[x index=5, y index=2]  {result_Laplacian_eigenvalues_1.dat};
      \addplot+[mark=square, color=red] table[x index=0, y index=2]  {result_Laplacian_eigenvalues_1.dat};
      
      \addplot+[mark=o, color=blue, mark options={solid}, only marks] table[x index=5, y index=1]  {result_Laplacian_eigenvalues_2.dat};
      \addplot+[mark=square, color=red] table[x index=0, y index=1]  {result_Laplacian_eigenvalues_2.dat};
      
      \addplot+[mark=o, color=blue, mark options={solid}, only marks] table[x index=6, y index=2]  {result_Laplacian_eigenvalues_2.dat};
      \addplot+[mark=square, color=red] table[x index=0, y index=2]  {result_Laplacian_eigenvalues_2.dat};
       \addplot+[mark=o, color=blue, mark options={solid}, only marks] table[x index=7, y index=3]  {result_Laplacian_eigenvalues_2.dat};
       \addplot+[mark=square, color=red] table[x index=0, y index=3]  {result_Laplacian_eigenvalues_2.dat};
     
       \addplot+[mark=o, color=blue, mark options={solid}, only marks] table[x index=8, y index=4]  {result_Laplacian_eigenvalues_2.dat};
       \addplot+[mark=square, color=red] table[x index=0, y index=4]  {result_Laplacian_eigenvalues_2.dat};
     
    \end{axis}          
  \end{tikzpicture}
}		&
 		{ 
 		\setlength\tabcolsep{0.8pt}
 		\begin{tabular}{cccccccc}
 		\vspace{-7cm}
 		&&&&&&&\\
 		 &&\multicolumn{6}{c}{{$k$}}\\
 		  \multicolumn{2}{c|}{{$\times10^{-3}$}}& 0 & 1 & 2 & 3 & 4 & 5 \\
 		 \hline
 		 &\multicolumn{1}{c|}{{1}} & -3.2 &&&&&\\
 		 & \multicolumn{1}{c|}{{2}} & &-10&&&&\\
 		 & \multicolumn{1}{c|}{{3}} & &&-8.4&&&\\
 		 & \multicolumn{1}{c|}{{4}} & &&-6.1&&&\\
 		 & \multicolumn{1}{c|}{{5}} & &&&-5.9&&\\
 		$m$ & \multicolumn{1}{c|}{{6}} & &&&-5.9&&\\
 		 & \multicolumn{1}{c|}{{7}} & &&&&-5.9&\\
 		 & \multicolumn{1}{c|}{{8}} & &&&&-5.9&\\
 		 & \multicolumn{1}{c|}{{9}} & &&&&&-5.9\\
 		 & \multicolumn{1}{c|}{{10}} & &&&&&-5.9\\
 		 & \multicolumn{1}{c|}{{11}} & &&&&&-5.9
 		\end{tabular}

}\\
	\end{tabular}
	\end{center}
  
\end{figure}
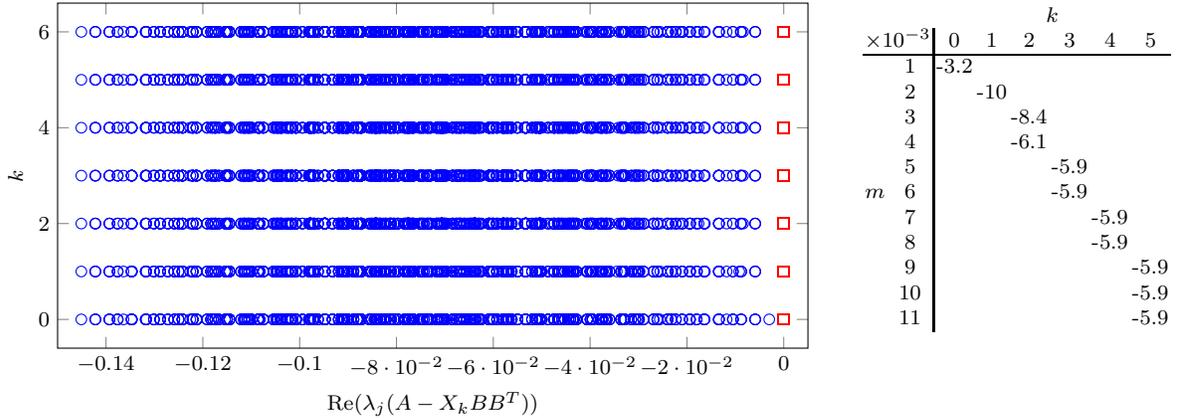

 }
\end{num_example}

\begin{num_example}\label{Ex.2}
 {\rm 
 We now consider the matrix $T\in\RR^{n\times n}$, $n=109460$, denominated {\tt lung} in the UF Sparse Matrix Collection \cite{Davis2011}. This unsymmetric matrix 
 has been used in \cite[Example 6]{Benner2018} as coefficient matrix of the Riccati equation \eqref{eq.Riccati}.
  However, $T$ is \emph{anti-stable}, i.e., the spectrum of $T$ is contained in the open right half plane. We thus consider $-T$ to our purpose. Even though $-T$ is stable, it is indefinite as 
  $\widebar \lambda:=\max_j(\lambda_j(-T-T^T)/2)>0$. Since we are not aware of any low-rank $X_0$ such that $-T-X_0BB^T<0$, we prefer to shift $-T$ and consider the negative definite matrix $A:=-T-(\widebar \lambda+1) I$. Moreover, the entries of $A$ have been reordered by means of the Matlab function {\tt symrcm}.  

  In Table \ref{tab2} we report the results for different values of $p$ and $q$.
  
 \begin{table}[!ht]
 \centering
 \caption{Example \ref{Ex.2}. Results for different values of $p$ and $q$. \label{tab2}}
\begin{tabular}{|r| r r r r r|}
 \hline
 & It. (max inner it.)  & Mem. & rank($X$) & Rel. Res & Time \\
 \hline
 &\multicolumn{5}{c|}{{$p=q=1$}} \\
\hline
 PNK\_EK & 36 (-)   & 74 &50  &9.82e-9 &3.69  \\

 PNK\_RK & 29 (-)  & 31&29& 7.62e-9&7.25 \\

 EKSM  & 36 (-)  &74 &60 &9.83e-9&3.40\\

 RKSM  & 29 (-)   &31 & 29 &9.07e-9&7.36\\
 
  iNK+ADI  & 4 (25)   &94 & 31 &8.01e-9&8.38\\

    NK+GP  & 2 (64)   &170 & 55 &4.40e-15&23.12\\

     RADI  & 33 (-)   &43 &33 &5.61e-9&5.81\\

\hline 

&\multicolumn{5}{c|}{{$p=q=10$}} \\
\hline
 PNK\_EK & 26 (-)   & 540 &369  &3.97e-11 &24.24  \\

 PNK\_RK & 28 (-)  &300&280& 2.07e-9&25.72 \\

 EKSM  &    24 (-)& 500& 433 & 5.19e-9 & 24.90\\

 RKSM  &  26 (-) &280 &260 &6.01e-9&22.53\\
 
  iNK+ADI  & 4 (30)   &1280 & 318 &5.06e-9&40.75\\
  
    NK+GP  & 2 (82)   &2300 &388 &1.08e-14&107.35\\

    RADI  & 31 (-)   &950 &310 &2.03e-9&24.54\\

\hline

&\multicolumn{5}{c|}{{$p=10$, $q=1$}} \\
\hline
 PNK\_EK & 48 (-)    & 98 &46  &1.02e-10 &5.62  \\

 PNK\_RK & 30 (-)  & 32&30& 3.72e-9&7.89 \\

 EKSM  &  33 (-)  &68 &55 &9.42e-9&3.36\\

 RKSM  & 28 (-)  &30 &28&8.80e-9&7.28\\
 
  iNK+ADI  & 4 (26)   &669 &34 &7.65e-9&24.80\\

    NK+GP  &  2 (62)  &1054 & 55 &2.69e-14&52.36\\

 RADI  &  27 (-)  &55 &27 &6.10e-9&8.16\\
    
\hline 

\end{tabular}
 \end{table}

PNK\_EK and EKSM are still among the fastest methods, especially for small $q$, and PNK\_RK, RKSM and RADI exhibit similar results, particularly in terms of computational time. For all the tested values of $p$ and $q$, both PNK\_EK and PNK\_RK implicitly solve four Lyapunov equations.
 
 We would like to underline how the computational cost of our new procedures does not really depend on $p$. More precisely, in Algorithm \ref{PNK_EK_algorithm} and \ref{PNK_RK_algorithm} we only solve $q$ linear systems per iteration, similarly to what is done in EKSM and RKSM. This does not hold for iNK+ADI, RADI and NK+GP.
 Indeed, in these methods, linear systems of the form $(A+\theta_jI+UV^T)Z=W$, $W\in\mathbb{R}^{n\times \ell}$, $U,V\in\mathbb{R}^{n\times p}$, have to be solved at each (inner) iteration. The number of columns $\ell$ of the right-hand side $W$ depends on the selected method. In particular, for iNK+ADI and NK+GP, $\ell=p+q$ so that, by employing the SMW formula, we solve $2p+q$ linear system at each inner iteration. In RADI, $W\in\mathbb{R}^{n\times q}$ and $p+q$ linear systems are solved at each iteration. See, e.g., \cite{Benner2018a} for more details.
  Therefore, if $p$ is large compared to $q$, the computational cost of iNK+ADI, RADI and NK+GP may dramatically increase while it remains almost constant in PNK\_EK and PNK\_RK.
 For instance, if we compare the performance of PNK\_RK for the cases $p=q=1$ and $p=10$, $q=1$ we obtain a similar number of iterations and basically the same computational time. On the other hand, the time of iNK+ADI and NK+GP is more than the double when $p=10$, $q=1$ compared to the case $p=q=1$. Also the computational time of RADI increases when $p=10$ and $q=1$ even though we perform fewer iterations compared to the case $p=q=1$.

 }

\end{num_example}

\section{Conclusions}\label{Conclusions}
A novel and effective approach for solving large-scale algebraic Riccati equations has been developed. The inexact Newton-Kleinman method has been combined with projection techniques that rely on timely approximation spaces like the extended and the rational Krylov subspaces. In our approach, only one approximation space is constructed as in the ``pure'' projection methods for matrix equations making our algorithm very efficient. The projected Newton-Kleinman procedures PNK\_EK and PNK\_RK perform very similarly to EKSM and RKSM respectively, in terms of both memory requirements and computational time. The numerical results show how our new algorithms are very competitive also with state-of-the-art procedures which are not based on projection.
Moreover, if $X$ is the computed solution, the stability of $A-XBB^T$ is ensured in our new framework if certain reasonable assumptions on the coefficient matrices hold.
Indeed, Theorem~\ref{Theorem_conv2} describes a large family of equations for which our schemes certainly compute a stabilizing solution whereas,
  to the best of our knowledge, no results available in the literature show that any other low-rank method for large-scale Riccati equations 
  is guaranteed to do the same. 
Furthermore, in our setting, a monotonic decrease in the Riccati residual norm has been shown thanks to the exact line search we perform.

 Both PNK\_EK and PNK\_RK can be easily modified to address the solution of generalized Riccati equations. However, preliminary results show that our routines are not competitive in terms of computational time with other state-of-the-art schemes. We believe that the performances of the projected Newton-Kleinman method applied to generalized Riccati equations can be largely improved if, e.g., nonstandard inner products are employed. This will be the topic of future works.

Another research direction is the solution of nonsymmetric Riccati equations \cite{Benner2016b,Bini2008}. Our new algorithms can be easily adapted to handle nonsymmetric problems and the solution process only require the construction of a \emph{right} and a \emph{left} subspace, in agreement with standard procedures for Sylvester equations. See, e.g., \cite[Section 4.4.1]{Simoncini2016}.


\section*{Acknowledgments}
We wish to thank Valeria Simoncini, Stefano Massei, Leonardo Robol and Patrick K\"{u}rschner for insightful conversations about earlier versions of the manuscript. Their helpful suggestions are greatly appreciated.
 We also thank Zvonimir Bujanovi\'c for certain clarifications about RADI and the two anonymous reviewers for their constructive remarks. 

 The author is a member of the Italian INdAM Research group GNCS.


\section*{Appendix}

Here we report the proof of Proposition~\ref{polynomial_coefficients_computation}.

\begin{proof}
 In this proof we only need the Arnoldi relation 
 \begin{equation}\label{Arnoldi_rel}
  AV_{\widebar m_{k+1}}=V_{\widebar m_{k+1}}T_{\widebar m_{k+1}}+\mathcal{V}_{\widebar m_{k+1}+1}E_{\widebar m_{k+1}+1}^T\underline{T}_{\widebar m_{k+1}},
 \end{equation}
 and the cyclic property of the trace operator, i.e.,
 $\mbox{trace}(AB)=\mbox{trace}(BA)$ for $A$ and $B$ matrices of conformal dimensions.

 We have
 $$\begin{array}{rll}
    \alpha_k&=&\|\mathcal{R}(X_k)\|^2_F=\|\mathcal{R}(V_{\widebar m_k}Y_{\widebar m_k}V_{\widebar m_k}^T)\|^2_F 
    \\
    
    &&\\
    
    &=& \|AV_{\widebar m_k}Y_{\widebar m_k}V_{\widebar m_k}^T+V_{\widebar m_k}Y_{\widebar m_k}V_{\widebar m_k}^TA^T-V_{\widebar m_k}Y_{\widebar m_k}V_{\widebar m_k}^TBB^TV_{\widebar m_k}Y_{\widebar m_k}V_{\widebar m_k}^T+C^TC\|_F^2 \\
    
    &&\\
    
    &=&\|V_{\widebar m_k}T_{\widebar m_k}Y_{\widebar m_k}V_{\widebar m_k}^T+V_{\widebar m_k}Y_{\widebar m_k}T_{\widebar m_k}^TV_{\widebar m_k}^T
    -V_{\widebar m_k}Y_{\widebar m_k}V_{\widebar m_k}^TBB^TV_{\widebar m_k}Y_{\widebar m_k}V_{\widebar m_k}^T+C^TC \\
    && +\mathcal{V}_{\widebar m_k+1}E_{\widebar m_k+1}^T\underline{T}_{\widebar m_k}Y_{\widebar m_k}V_{\widebar m_k}^T+V_{\widebar m_k}Y_{\widebar m_k}\underline{T}_{\widebar m_k}^TE_{\widebar m_k+1}\mathcal{V}_{\widebar m_k+1}^T\|^2_F \\
    
    &&\\
    
    &=&\|V_{\widebar m_k}\left(T_{\widebar m_k}Y_{\widebar m_k}+Y_{\widebar m_k}T_{\widebar m_k}^T
    -Y_{\widebar m_k}B_{\widebar m_k}B_{\widebar m_k}^TY_{\widebar m_k}+E_1\pmb{\gamma\gamma}^TE_1^T\right)V_{\widebar m_k}^T \\
    && +\mathcal{V}_{\widebar m_k+1}E_{\widebar m_k+1}^T\underline{T}_{\widebar m_k}Y_{\widebar m_k}V_{\widebar m_k}^T+V_{\widebar m_k}Y_{\widebar m_k}\underline{T}_{\widebar m_k}^TE_{\widebar m_k+1}\mathcal{V}_{\widebar m_k+1}^T\|^2_F. \\
    
%
    \end{array}
$$
Since $\langle \mathcal{V}_{\widebar m_k+1}, V_{\widebar m_k}\rangle_F=0$ by construction, we have

 $$\begin{array}{rll}
    \alpha_k&=& \|T_{\widebar m_k}Y_{\widebar m_k}+Y_{\widebar m_k}T_{\widebar m_k}^T
    -Y_{\widebar m_k}B_{\widebar m_k}B_{\widebar m_k}^TY_{\widebar m_k}+E_1\pmb{\gamma\gamma}^TE_1^T\|_F^2 \\
    && +\|\mathcal{V}_{\widebar m_k+1}E_{\widebar m_k+1}^T\underline{T}_{\widebar m_k}Y_{\widebar m_k}V_{\widebar m_k}^T+V_{\widebar m_k}Y_{\widebar m_k}\underline{T}_{\widebar m_k}^TE_{\widebar m_k+1}\mathcal{V}_{\widebar m_k+1}^T\|^2_F \\
    
     &&\\
&=& \|T_{\widebar m_k}Y_{\widebar m_k}+Y_{\widebar m_k}T_{\widebar m_k}^T
    -Y_{\widebar m_k}B_{\widebar m_k}B_{\widebar m_k}^TY_{\widebar m_k}+E_1\pmb{\gamma\gamma}^TE_1^T\|_F^2+2\|Y_{\widebar m_k}\underline{T}_{\widebar m_k}^TE_{\widebar m_k+1}\|^2_F. \\
    \end{array}
$$
Moreover, recalling that
 $\|L_{k+1}\|_F=\sqrt{2}\|\widetilde Y_{\widebar m_{k+1}}\underline{T}_{\widebar m_{k+1}}^TE_{\widebar m_{k+1}+1}\|_F$, it holds
$$\beta_k=\|L_{k+1}\|_F^2=2\|\widetilde Y_{\widebar m_{k+1}}\underline{T}_{\widebar m_{k+1}}^TE_{\widebar m_{k+1}+1}\|_F^2.$$
Then
$$
\begin{array}{rll}
 \gamma_k&=&\langle \mathcal{R}(X_k),L_{k+1}\rangle_F\\
 
 &&\\
 
 &=&
 \langle AV_{\widebar m_k}Y_{\widebar m_k}V_{\widebar m_k}^T+V_{\widebar m_k}Y_{\widebar m_k}V_{\widebar m_k}^TA^T
  -V_{\widebar m_k}Y_{\widebar m_k}V_{\widebar m_k}^TBB^TV_{\widebar m_k}Y_{\widebar m_k}V_{\widebar m_k}^T+C^TC,\\
  &&(A-V_{\widebar m_k}Y_{\widebar m_k}V_{\widebar m_k}^TBB^T)
  V_{\widebar m_{k+1}}\widetilde Y_{\widebar m_{k+1}}V_{\widebar m_{k+1}}^T +V_{\widebar m_{k+1}}\widetilde Y_{\widebar m_{k+1}}V_{\widebar m_{k+1}}^T(A-V_{\widebar m_k}Y_{\widebar m_k}V_{\widebar m_k}^TBB^T)^T \\
 &&+
V_{\widebar m_k}Y_{\widebar m_k}V_{\widebar m_k}^TBB^TV_{\widebar m_k}Y_{\widebar m_k}V_{\widebar m_k}^T+C^TC\rangle_F\\
 
 &&\\
 
 &=& \langle V_{\widebar m_k}\left(T_{\widebar m_k}Y_{\widebar m_k}+Y_{\widebar m_k}T_{\widebar m_k}^T
    -Y_{\widebar m_k}B_{\widebar m_k}B_{\widebar m_k}^TY_{\widebar m_k}+E_1\pmb{\gamma\gamma}^TE_1^T\right)V_{\widebar m_k}^T \\
    && +\mathcal{V}_{\widebar m_k+1}E_{\widebar m_k+1}^T\underline{T}_{\widebar m_k}Y_{\widebar m_k}V_{\widebar m_k}^T+V_{\widebar m_k}Y_{\widebar m_k}\underline{T}_{\widebar m_k}^TE_{\widebar m_k+1}\mathcal{V}_{\widebar m_k+1}^T, \\
 &&\mathcal{V}_{\widebar m_{k+1}+1}E_{\widebar m_{k+1}+1}^T\underline{T}_{\widebar m_{k+1}}
 \widetilde Y_{\widebar m_{k+1}}V_{\widebar m_{k+1}}^T
  +V_{\widebar m_{k+1}}\widetilde Y_{\widebar m_{k+1}}\underline{T}_{\widebar m_{k+1}}^TE_{\widebar m_{k+1}+1}
\mathcal{V}_{\widebar m_{k+1}+1}^T\rangle_F\\

&&\\

 &=&\langle V_{\widebar m_k}\left(T_{\widebar m_k}Y_{\widebar m_k}+Y_{\widebar m_k}T_{\widebar m_k}^T
    -Y_{\widebar m_k}B_{\widebar m_k}B_{\widebar m_k}^TY_{\widebar m_k}+E_1\pmb{\gamma\gamma}^TE_1^T\right)V_{\widebar m_k}^T,\\
 &&\mathcal{V}_{\widebar m_{k+1}+1}E_{\widebar m_{k+1}+1}^T\underline{T}_{\widebar m_{k+1}}
 \widetilde Y_{\widebar m_{k+1}}V_{\widebar m_{k+1}}^T  +V_{\widebar m_{k+1}}\widetilde Y_{\widebar m_{k+1}}\underline{T}_{\widebar m_{k+1}}^TE_{\widebar m_{k+1}+1}
\mathcal{V}_{\widebar m_{k+1}+1}^T\rangle_F\\
    &&+ \langle\mathcal{V}_{\widebar m_k+1}E_{\widebar m_k+1}^T\underline{T}_{\widebar m_k}Y_{\widebar m_k}V_{\widebar m_k}^T+V_{\widebar m_k}Y_{\widebar m_k}\underline{T}_{\widebar m_k}^TE_{\widebar m_k+1}\mathcal{V}_{\widebar m_k+1}^T, \\
 &&\mathcal{V}_{\widebar m_{k+1}+1}E_{\widebar m_{k+1}+1}^T\underline{T}_{\widebar m_{k+1}}
 \widetilde Y_{\widebar m_{k+1}}V_{\widebar m_{k+1}}^T
  +V_{\widebar m_{k+1}}\widetilde Y_{\widebar m_{k+1}}\underline{T}_{\widebar m_{k+1}}^TE_{\widebar m_{k+1}+1}
\mathcal{V}_{\widebar m_{k+1}+1}^T\rangle_F\\

%
\end{array}
$$
The first inner product in the above expression is zero due to the orthogonality of $V_{\widebar m_{k}}$ and $\mathcal{V}_{\widebar m_{k+1}+1}$. The same happens also to the second term if $\widebar m_{k+1}>\widebar m_{k}$.
If $\widebar m_{k+1}=\widebar m_{k}$ instead, we have
$$
\begin{array}{rll}
 \gamma_k &=& \langle\mathcal{V}_{\widebar m_k+1}E_{\widebar m_k+1}^T\underline{T}_{\widebar m_k}Y_{\widebar m_k}V_{\widebar m_k}^T+V_{\widebar m_k}Y_{\widebar m_k}\underline{T}_{\widebar m_k}^TE_{\widebar m_k+1}\mathcal{V}_{\widebar m_k+1}^T, \\
 &&\mathcal{V}_{\widebar m_{k+1}+1}E_{\widebar m_{k+1}+1}^T\underline{T}_{\widebar m_{k+1}}
 \widetilde Y_{\widebar m_{k+1}}V_{\widebar m_{k+1}}^T
  +V_{\widebar m_{k+1}}\widetilde Y_{\widebar m_{k+1}}\underline{T}_{\widebar m_{k+1}}^TE_{\widebar m_{k+1}+1}
\mathcal{V}_{\widebar m_{k+1}+1}^T\rangle_F\\

&&\\

&=&\langle \mathcal{V}_{\widebar m_k+1}E_{\widebar m_k+1}^T\underline{T}_{\widebar m_k}Y_{\widebar m_k}V_{\widebar m_k}^T,\mathcal{V}_{\widebar m_{k+1}+1}E_{\widebar m_{k+1}+1}^T\underline{T}_{\widebar m_{k+1}}
 \widetilde Y_{\widebar m_{k+1}}V_{\widebar m_{k+1}}^T\rangle_F \\
&&+\langle V_{\widebar m_k}Y_{\widebar m_k}\underline{T}_{\widebar m_k}^TE_{\widebar m_k+1}\mathcal{V}_{\widebar m_k+1}^T,V_{\widebar m_{k+1}}\widetilde Y_{\widebar m_{k+1}}\underline{T}_{\widebar m_{k+1}}^TE_{\widebar m_{k+1}+1}
\mathcal{V}_{\widebar m_{k+1}+1}^T\rangle_F\\
&&\\
&=&2 \langle E_{\widebar m_k+1}^T\underline{T}_{\widebar m_k}Y_{\widebar m_k}, 
 E_{\widebar m_{k+1}+1}^T\underline{T}_{\widebar m_{k+1}}
 \widetilde Y_{\widebar m_{k+1}}
\rangle_F.
\end{array}
$$
Recalling that 
$$Z_k=\widetilde X_{k+1}-X_k=V_{\widebar m_{k+1}}\widetilde Y_{\widebar m_{k+1}}V_{\widebar m_{k+1}}^T-
V_{\widebar m_{k}}Y_{\widebar m_{k}}V_{\widebar m_{k}}^T=
V_{\widebar m_{k+1}}(\widetilde Y_{\widebar m_{k+1}}-\mbox{diag}(Y_{\widebar m_{k}},O_{2q(\widebar m_{k+1}-\widebar m_{k})}))V_{\widebar m_{k+1}}^T,$$
we have 
\begin{align*}
   \delta_k=&\|Z_kBB^TZ_k\|_F^2\\
   =&\|V_{\widebar m_{k+1}}\left(\widetilde Y_{\widebar m_{k+1}}-\mbox{diag}(Y_{\widebar m_{k}},O_{2q(\widebar m_{k+1}-\widebar m_{k})})\right)V_{\widebar m_{k+1}}^TB
    B^TV_{\widebar m_{k+1}}\left(\widetilde Y_{\widebar m_{k+1}}- \mbox{diag}(Y_{\widebar m_{k}},O_{2q(\widebar m_{k+1}-\widebar m_{k})})\right)V_{\widebar m_{k+1}}^T\|^2_F \\
=&
\|\left(\widetilde Y_{\widebar m_{k+1}}-\mbox{diag}(Y_{\widebar m_{k}},O_{2q(\widebar m_{k+1}-\widebar m_{k})})\right)B_{\widebar m_{k+1}}B_{\widebar m_{k+1}}^T
\left(\widetilde Y_{\widebar m_{k+1}}-\mbox{diag}(Y_{\widebar m_{k}},O_{2q(\widebar m_{k+1}-\widebar m_{k})}\right)\|^2_F.
\end{align*}
Moreover,
$$
\begin{array}{rll}
 \epsilon_k&=&\langle \mathcal{R}(X_k),Z_kBB^TZ_k\rangle_F \\
 
 &&\\
 &=& \langle V_{\widebar m_k}\left(T_{\widebar m_k}Y_{\widebar m_k}+Y_{\widebar m_k}T_{\widebar m_k}^T
    -Y_{\widebar m_k}B_{\widebar m_k}B_{\widebar m_k}^TY_{\widebar m_k}+E_1\pmb{\gamma\gamma}^TE_1^T\right)V_{\widebar m_k}^T \\
    && +\mathcal{V}_{\widebar m_k+1}E_{\widebar m_k+1}^T\underline{T}_{\widebar m_k}Y_{\widebar m_k}V_{\widebar m_k}^T+V_{\widebar m_k}Y_{\widebar m_k}\underline{T}_{\widebar m_k}^TE_{\widebar m_k+1}\mathcal{V}_{\widebar m_k+1}^T, \\
 &&V_{\widebar m_{k+1}}\left(\widetilde Y_{\widebar m_{k+1}}-\mbox{diag}(Y_{\widebar m_{k}},O_{2q(\widebar m_{k+1}-\widebar m_{k})})\right)B_{\widebar m_{k+1}}
    B_{\widebar m_{k+1}}^T\left(\widetilde Y_{\widebar m_{k+1}}- \mbox{diag}(Y_{\widebar m_{k}},O_{2q(\widebar m_{k+1}-\widebar m_{k})})\right)V_{\widebar m_{k+1}}^T\rangle_F\\

    &&\\
 &=& \langle V_{\widebar m_k}\left(T_{\widebar m_k}Y_{\widebar m_k}+Y_{\widebar m_k}T_{\widebar m_k}^T
    -Y_{\widebar m_k}B_{\widebar m_k}B_{\widebar m_k}^TY_{\widebar m_k}+E_1\pmb{\gamma\gamma}^TE_1^T\right)V_{\widebar m_k}^T, \\
 &&V_{\widebar m_{k+1}}\left(\widetilde Y_{\widebar m_{k+1}}-\mbox{diag}(Y_{\widebar m_{k}},O_{2q(\widebar m_{k+1}-\widebar m_{k})})\right)B_{\widebar m_{k+1}}
    B_{\widebar m_{k+1}}^T\left(\widetilde Y_{\widebar m_{k+1}}- \mbox{diag}(Y_{\widebar m_{k}},O_{2q(\widebar m_{k+1}-\widebar m_{k})})\right)V_{\widebar m_{k+1}}^T\rangle_F\\      
    && +\langle \mathcal{V}_{\widebar m_k+1}E_{\widebar m_k+1}^T\underline{T}_{\widebar m_k}Y_{\widebar m_k}V_{\widebar m_k}^T+V_{\widebar m_k}Y_{\widebar m_k}\underline{T}_{\widebar m_k}^TE_{\widebar m_k+1}\mathcal{V}_{\widebar m_k+1}^T, \\
 &&V_{\widebar m_{k+1}}\left(\widetilde Y_{\widebar m_{k+1}}-\mbox{diag}(Y_{\widebar m_{k}},O_{2q(\widebar m_{k+1}-\widebar m_{k})})\right)B_{\widebar m_{k+1}}
    B_{\widebar m_{k+1}}^T\left(\widetilde Y_{\widebar m_{k+1}}- \mbox{diag}(Y_{\widebar m_{k}},O_{2q(\widebar m_{k+1}-\widebar m_{k})})\right)V_{\widebar m_{k+1}}^T\rangle_F.\\
\end{array}
$$
If $\widebar m_{k+1}=\widebar m_{k}$, the second inner product above is zero while the first can be written as $\langle T_{\widebar m_k}Y_{\widebar m_k}+Y_{\widebar m_k}T_{\widebar m_k}^T
    -Y_{\widebar m_k}B_{\widebar m_k}B_{\widebar m_k}^TY_{\widebar m_k}+E_1\pmb{\gamma\gamma}^TE_1^T, (\widetilde Y_{\widebar m_{k+1}}-Y_{\widebar m_{k}})B_{\widebar m_{k+1}}
    B_{\widebar m_{k+1}}^T(\widetilde Y_{\widebar m_{k+1}}-Y_{\widebar m_{k}})\rangle_F$.
If $\widebar m_{k+1}>\widebar m_{k}$, also the second term in the above expression must be taken into account and this can be written as 
\begin{multline*}
 \langle V_{\widebar m_k+1}\left(E_{\widebar m_k+1}E_{\widebar m_k+1}^T\underline{T}_{\widebar m_k}[Y_{\widebar m_k},O_{2q\widebar m_k\times 2q}]+[Y_{\widebar m_k};O_{2q\times 2q\widebar m_k}]\underline{T}_{\widebar m_k}^TE_{\widebar m_k+1}E_{\widebar m_k+1}^T\right)V_{\widebar m_k+1}^T, \\
 V_{\widebar m_{k+1}}\left(\widetilde Y_{\widebar m_{k+1}}-\mbox{diag}(Y_{\widebar m_{k}},O_{2q(\widebar m_{k+1}-\widebar m_{k})})\right)B_{\widebar m_{k+1}}
    B_{\widebar m_{k+1}}^T\left(\widetilde Y_{\widebar m_{k+1}}- \mbox{diag}(Y_{\widebar m_{k}},O_{2q(\widebar m_{k+1}-\widebar m_{k})})\right)V_{\widebar m_{k+1}}^T\rangle_F,
\end{multline*}
that is
{\small 
\begin{multline*}
 \langle E_{\widebar m_k+1}E_{\widebar m_k+1}^T\underline{T}_{\widebar m_k}[Y_{\widebar m_k},O_{2q\widebar m_k\times 2q}]+[Y_{\widebar m_k};O_{2q\times 2q\widebar m_k}]\underline{T}_{\widebar m_k}^TE_{\widebar m_k+1}E_{\widebar m_k+1}^T, \\
 V_{\widebar m_k+1}^TV_{\widebar m_{k+1}}\left(\widetilde Y_{\widebar m_{k+1}}-\mbox{diag}(Y_{\widebar m_{k}},O_{2q(\widebar m_{k+1}-\widebar m_{k})})\right)B_{\widebar m_{k+1}}
    B_{\widebar m_{k+1}}^T\left(\widetilde Y_{\widebar m_{k+1}}- \mbox{diag}(Y_{\widebar m_{k}},O_{2q(\widebar m_{k+1}-\widebar m_{k})})\right)V_{\widebar m_{k+1}}^TV_{\widebar m_k+1}\rangle_F.
\end{multline*}
}
Since 
$$V_{\widebar m_k}^TV_{\widebar m_{k+1}}=[I_{2q\widebar m_k},O_{2q\widebar m_k\times 2q(\widebar m_{k+1}-\widebar m_k)}],$$ 
and
$$
V_{\widebar m_k+1}^TV_{\widebar m_{k+1}}=[I_{2q(\widebar m_k+1)},O_{2q(\widebar m_k+1)\times 2q(\widebar m_{k+1}-\widebar m_k-1)}],$$
we get the result for $\epsilon_k$.

To conclude,
$$
\begin{array}{rll}
 \zeta_k&=&\langle L_{k+1}, Z_kBB^TZ_k\rangle_F\\

 &&\\
 
 &=& \langle \mathcal{V}_{\widebar m_{k+1}+1}E_{\widebar m_{k+1}+1}^T\underline{T}_{\widebar m_{k+1}}\widetilde Y_{\widebar m_{k+1}}V_{\widebar m_{k+1}}^T+V_{\widebar m_{k+1}}\widetilde Y_{\widebar m_{k+1}}\underline{T}_{\widebar m_{k+1}}^TE_{\widebar m_{k+1}+1}\mathcal{V}_{\widebar m_{k+1}+1}^T, \\
 &&V_{\widebar m_{k+1}}\left(\widetilde Y_{\widebar m_{k+1}}-\mbox{diag}(Y_{\widebar m_{k}},O_{2q(\widebar m_{k+1}-\widebar m_{k})})\right)B_{\widebar m_{k+1}}
    B_{\widebar m_{k+1}}^T\left(\widetilde Y_{\widebar m_{k+1}}- \mbox{diag}(Y_{\widebar m_{k}},O_{2q(\widebar m_{k+1}-\widebar m_{k})})\right)V_{\widebar m_{k+1}}^T\rangle_F\\
&&\\

 &=&0,\\
\end{array}
$$
since $\mathcal{V}_{\widebar m_{k+1}+1}^TV_{\widebar m_{k+1}}=O_{2q\times 2q\widebar m_{k+1}}$.
\begin{flushright}
 $\square$
\end{flushright}
\end{proof}

\bibliography{LowRankRiccati}

\end{document}

%% file: generic.tex
\usepackage{amsmath,amssymb}
\usepackage{cite}
\usepackage{mathtools}
\usepackage{tikz}
\usepackage{pgfplots}
\usepackage{pgfplotstable}
\usepackage{geometry}
\usepackage{color}
\usepackage{booktabs}
\usepackage{framed}
\pgfplotsset{compat=1.9}

\pgfplotstableset{
	every head row/.style={before row=\toprule,after row=\midrule},
	clear infinite
}

\usepackage{amsopn}

\renewcommand{\leq}{\leqslant}
\renewcommand{\geq}{\geqslant}